\documentclass[12pt,reqno]{article}
\usepackage[letterpaper,margin=1.6cm]{geometry}
\usepackage{latexsym, amsfonts, amsthm,amssymb,amscd,amsmath,makeidx}
\usepackage{enumerate}
\usepackage[normalem]{ulem}
\usepackage{exscale}
\usepackage{overpic} 
\usepackage{color} 
\usepackage{graphicx}
\usepackage{bm}
\usepackage{comment}
\usepackage{caption}
\usepackage{float}
\usepackage{array} 
\usepackage{booktabs} 
\usepackage{tabularx}
\usepackage{cancel}
\usepackage{mathrsfs}
\usepackage{amssymb}
\usepackage{mathtools}
\DeclarePairedDelimiter\ceil{\lceil}{\rceil}
\DeclarePairedDelimiter\floor{\lfloor}{\rfloor}

\usepackage{tikz}
\usepackage{blindtext}
\usepackage{bbm}
\usepackage{fancyhdr} 
\usepackage{mathrsfs}
\usepackage{wrapfig}
\usepackage{hyperref}
\usepackage{bookmark}
\usepackage{cleveref}
\usepackage{cases}
\usepackage{multirow}
\usepackage{pdflscape}
\usepackage[title]{appendix}

\hypersetup{ colorlinks, citecolor=red, filecolor=black, linkcolor=blue,
	urlcolor=black }

\definecolor{DarkGreen}{rgb}{0.2,0.6,0.2}

\def\eps{\varepsilon} \def\Om{\Omega}\def\om{\omega}

\def\Ind#1{{\mathbbmss 1}_{_{\scriptstyle #1}}}

\newcommand{\norm}[1]{\left\lVert {#1}\right\rVert}
\newcolumntype{Y}{>{\centering\arraybackslash}X} 
\def\lra{\longrightarrow}  
 \def\ua{\uparrow} \def\da{\downarrow}

\def\wh{\widehat} \def\wt{\widetilde}

\newcolumntype{C}{>{\centering\arraybackslash}X}

 \def\ignore#1{}

\def\bR{{\mathbb R}}

 \def\bZ{\mathbb Z} \def\bN{\mathbb N} \def\bQ{\mathbb Q}
\def\bT{\mathbb T}   
\def\bP{{\mathbb P}} \def\bE{{\mathbb E}}

\numberwithin{equation}{section}

\usetikzlibrary{positioning,calc,cd}

\allowdisplaybreaks[4]

\def\cB{{\mathscr B}}

 \def\cF{{\mathscr F}} \def\cG{{\mathscr G}}

\newtheorem{theorem}{Theorem}[section]
\newtheorem{proposition}[theorem]{Proposition}
\newtheorem{lemma}[theorem]{Lemma}
\newtheorem{corollary}[theorem]{Corollary}

\theoremstyle{definition}
\newtheorem{definition}[theorem]{Definition}
\newtheorem{example}[theorem]{Example}

\newtheorem{remark}[theorem]{Remark}

\def\Ind#1{{\mathbbmss 1}_{_{\scriptstyle #1}}} \def\lra{\longrightarrow}
\def\eps{\varepsilon} \def\<{\langle} \def\>{\rangle} \def\wt#1{\widetilde{#1}}

 \def\lra{\longrightarrow} 
  \def\ua{\uparrow} \def\da{\downarrow}

\def\wh{\widehat} \def\wt{\widetilde}  
\def\trace#1{\text{trace}\, #1}
 
  \def\argmin{\mathop{\hbox{\rm arg\,min}}}

\begin{document}
	
	\title{The roughness exponent and its model-free estimation}
	\author{Xiyue Han$^*$ and Alexander Schied\thanks{ University of Waterloo,
 200 University Ave W, Waterloo, Ontario, N2L 3G1, Canada. E-Mails: {\tt
 xiyue.han@uwaterloo.ca, aschied@uwaterloo.ca}.\hfill\break
The authors gratefully acknowledge support from the Natural Sciences and
 Engineering Research Council of Canada through grant
 RGPIN-2017-04054. }} \date{  \normalsize  First version: November 19, 2021\\ \normalsize This version: June 21, 2024 }
		
	\maketitle
	\begin{abstract}
	 Motivated by pathwise stochastic calculus, we say that a  continuous
real-valued function $x$ admits the roughness  exponent  $R$ if  the
$p^{\text{th}}$ variation of $x$ converges to zero for $p>1/R$ and to infinity for
$p<1/R$. 
 In our main result, we provide a mild condition on the
Faber--Schauder coefficients of  $x$ under which the roughness exponent
exists and is given as the limit of the classical Gladyshev estimates $\wh
R_n(x)$. This result can be  viewed as a strong consistency result for the
Gladyshev estimators in an entirely model-free setting, because it works strictly trajectory-wise and requires no probabilistic assumptions. Nonetheless,
our proof is probabilistic and relies on a martingale  hidden in the
Faber--Schauder expansion of $x$. We show that the condition of our main result is satisfied for the typical sample paths of fractional Brownian motion with drift, and we provide almost-sure convergence rates for the corresponding Gladyshev estimates. We also discuss the connections between the roughness exponent and 
the related concepts of Besov regularity and weighted quadratic variation.  Since the Gladyshev estimators are not
scale-invariant, we construct several scale-invariant estimators. 
Finally, we extend our results to the case in which  the $p^{\text{th}}$
variation of $x$ is defined over a sequence of unequally spaced partitions. 
	\end{abstract}

\smallskip
\noindent \textbf{MSC2020 subject classifications:} 60F15, 60G22, 60G46, 62G05, 26A30


	\section{Introduction}

The Hurst parameter was originally defined by Hurst~\cite{Hurst} as a
measure of the autocorrelation of a time series. But it is well known that
it can also determine a degree of \lq roughness\rq\ of the trajectories of certain stochastic processes, such as fractional Brownian motion. However,
Gneiting and Schlather~\cite{Gneiting2004Hurst} constructed a class of
stationary Gaussian processes for which the  Hurst parameter and its roughness decouple completely, if the roughness is quantified in terms of fractal
dimension; see also Cheridito \cite{CheriditoDiss} and Bennedsen et al.~\cite{BennedsenLundePakkanen} for related work. 	It is therefore necessary to distinguish
between the classical, autocorrelation-based Hurst parameter and a suitable
index for the roughness of a trajectory.  	In this paper, we study such a
roughness index, which is based on the 	$p^{\text{th}}$ variation of a
continuous real-valued function. 
	
	Recall that, if $x:[0,1]\to\bR$ is a continuous function and $p>0$, the
	$p^{\text{th}}$ variation of $x$ along the sequence of dyadic partitions is
	defined as the limit of 
	$$
		\<x\>_n^{(p)} := \sum_{k = 0}^{2^n-1}\big|x((k+1)2^{-n}) - x(k2^{-n})\big|^p,
	$$
	provided that this limit exists. We will say that $x$ admits the \emph{roughness exponent} $R \in [0,1]$ if 
	\begin{equation}\label{Hurst 1}
	\lim_{n \ua \infty}\<x\>_n^{(p)}  = \begin{cases}
			0 &\quad \text{for $p > 1/R$},\\
			\infty &\quad \text{for $p <1/R$}.
		\end{cases}
	\end{equation}
	Intuitively, the smaller $R$, the rougher the trajectory $x$ will look. For
	instance, if $x$ is continuously differentiable, then \eqref{Hurst 1} 
	holds with $R=1$, and this is also the largest possible value for $R$ unless
	$x$ is constant.  If $x$ is a typical sample path of a continuous
	semimartingale such as Brownian motion, then \eqref{eq_Hurst} holds with
	$R=1/2$. If $x$ is a typical sample path of fractional Brownian motion,
	then $R$  is equal to its classical Hurst parameter.
	
When measuring the roughness of a trajectory, there are strong reasons for
favouring $p^{\text{th}}$ variation over other measures such as fractal
dimension. First, the $p^{\text{th}}$ variation can be measured in a
straightforward manner. Second, other measures may lead to diverging results,
and there is no canonical choice. For instance, the graph of a trajectory may
have different Hausdorff and box dimensions. Third, and probably most
importantly, the $p^{\text{th}}$ variation plays a crucial role in the extension
of It\^o calculus to rough trajectories. A pathwise and strictly model-free
version of such an extension for integrators with arbitrary $p^{\text{th}}$
variation was recently established by Cont and Perkowski~\cite{ContPerkowski}.
It\^o calculus for rough trajectories also plays an important role in
applications, for instance to rough volatility models. These models are based on
the observation by Gatheral et al.~\cite{GatheralRosenbaum} that the Hurst
parameter of the realized volatility of many financial time series is
rather small, which makes realized volatility much rougher than the sample paths
of a continuous semimartingale.  

In \Cref{Thm_main}, our main result, we establish the existence of the roughness exponent under a mild condition on the Faber--Schauder coefficients
 of $x$. We call this condition the reverse Jensen condition.  Even more important than the existence of the roughness exponent
 is the fact that  the reverse Jensen condition guarantees that the 
 roughness exponent of $x$ can be obtained as the limit of the classical
 Gladyshev estimates $\wh R_n(x)$.  This part of \Cref{Thm_main} can be viewed
 as a strong consistency result for the Gladyshev estimators in an entirely
 model-free setting, because it works strictly trajectory-wise and requires no probabilistic assumptions. In particular, we do not assume that $x$ can be
 obtained as a sample path of some stochastic process, even though this is clearly the typical application we have in mind. While the statement of \Cref{Thm_main} appears to be deterministic,
its proof is probabilistic. It relies on martingale techniques that are applied
to a martingale that is hidden in the Faber--Schauder expansion of $x$. The corresponding theory is developed in \Cref{martingale section}, where also the proof of \Cref{Thm_main}  is given.
 
 The Gladyshev estimator has traditionally been used to estimate the 
  Hurst parameter of stochastic processes such as fractional Brownian motion.
  So one conceptual contribution of \Cref{Thm_main} is the observation that the
  Gladyshev estimator actually estimates the roughness exponent and that
  its limit only delivers the classical Hurst parameter if and only if the
  latter coincides with the former. This is the case for fractional Brownian motion, which is the subject of \Cref{Gaussian examples section}. Our main result in this section, \Cref{Theorem example FBM}, shows in particular that the typical sample paths of fractional Brownian motion with drift satisfy the reverse Jensen condition. Additionally, it provides almost-sure convergence rates for the corresponding Gladyshev estimates.
  
 The Gladyshev estimator was originally derived from Gladyshev's theorem \cite{Gladyshev}, which is an almost-sure limit theorem for the weighted quadratic variation of certain Gaussian processes. In \Cref{weighted QV section} we systematically investigate the relations between weighted quadratic variation and the roughness exponent. In particular, we show that for any continuous function the existence of the weighted quadratic variation is strictly stronger than the convergence of the Gladyshev estimates. However, we also provide an example of a function  for which the parameter derived from the weighted quadratic variation is different from the roughness exponent. 
  
Due to a result by Rosenbaum~\cite{Rosenbaum2009}, the existence of the weighted quadratic variation is in turn closely related to the concept of Besov regularity. In  \Cref{Besov section}, we systematically analyze relations between  variants of this concept and the existence of a roughness parameter. First, we observe that  Eyink's formulation of  Besov regularity  \cite{EyinkBesov} leads  to a roughness concept that is generally different from ours. Then we examine somewhat more restrictive requirement, which is sometimes proposed in the rough volatility context. Our main result in this section, \Cref{Theorem Besov}, links the corresponding roughness measure to $p^{\text{th}}$ variations. However, \Cref{example Besov} shows that even this concept of Besov regularity does not imply the existence of a roughness exponent.

The results in Sections \ref{weighted QV section} and \ref{Besov section} hold also if the reverse Jensen condition is not assumed. 
In \Cref{gen section} we collect several additional results that are also valid in this more general context.

In \Cref{stat section}, we discuss the problem of estimating the roughness
exponent from discrete observations of a given function $x$. Here, the fact that the Gladyshev estimator is not
scale-invariant can become an issue. We therefore provide several families of
scale-invariant estimators that are derived from the sequence  $(\wh
R_n)_{n\in\bN}$ without deteriorating the corresponding rate of convergence. We also relate the estimator used in~\cite{GatheralRosenbaum}
to those families.

 In \Cref{irregular
section}, we provide an extension of our main results to the case in which the
$p^{\text{th}}$ variation is defined over a sequence of unequally spaced
partitions.

	\section{The roughness exponent: definition and
	existence}\label{general section}
	
	Let $x:[0,1]\to\bR$ be any fixed  continuous function. This function can be a natural or economic time series, a typical sample path of
	a stochastic process, or a fractal function. What these phenomena have in
	common is that the corresponding trajectories are not smooth but exhibit a
	certain degree of \lq roughness\rq. In the sequel, our goal is to quantify
	and measure the degree of that roughness. To this end, we will henceforth
	exclude the trivial case of a constant function $x$.

	For  $p >0$ and $n \in \bN$, we define 	\begin{equation}\label{Vnpx}
		\<x\>_n^{(p)} := \sum_{k = 0}^{2^n-1}\big|x((k+1)2^{-n}) -
		x(k2^{-n})\big|^p,
	\end{equation}
which can be regarded as the $p^{\text{\rm th}}$ variation of the function $x$
sampled along the dyadic partition
$\{k2^{-n}:k=0,\dots, 2^n\}$. Suppose that there exists $q $ such that
	\begin{equation}\label{eq_Hurst}
		\lim_{n \ua \infty}\<x\>_n^{(p)}  = \begin{cases}
			0 &\quad \text{for $p > q$},\\
			\infty &\quad \text{for $p <q$}.
		\end{cases}
	\end{equation}
Intuitively, the larger $q$, the rougher the  trajectory $x$ will look. For
 instance, if $x$ is continuously differentiable, then \eqref{eq_Hurst}  
 holds with $q=1$, and this is also the smallest possible value for $q$, because
 $x$ is nonconstant by assumption. 
Moreover, it is easy to see that \eqref{Hurst 1} holds if $\lim_n\<x\>^{(p)}_n$ exists in $(0,\infty)$ (see, e.g., the final step in the proof of Theorem 2.1 in~\cite{MishuraSchied2}). In particular, if
 $x$ is a typical sample path of a
 continuous semimartingale such as Brownian motion, then \eqref{eq_Hurst} holds with $q=2$. More generally, if $x$ is a typical sample path of a
 fractional Brownian motion with Hurst parameter $H\in(0,1)$,  then $q$  is
 equal to $1/H$; see \Cref{Theorem example FBM}  and the subsequent paragraph. The extreme case $q = \infty$ implies that $\<x\>^{(p)}_n$ diverges to infinity for any $p \ge 1$; an example is provided in \Cref{H=0 example}.

 \begin{definition}\label{def_Hurst}
 Suppose that there exists $q\in[1,\infty]$
 such that \eqref{eq_Hurst} holds. Then $R:=1/q$ is called the \emph{roughness exponent} of $x\in C[0,1]$.
 \end{definition}

A concept closely related to the roughness exponent was used by Gatheral
et al.~\cite{GatheralRosenbaum} to quantify the roughness of empirical
volatility time series. In~\cite{MishuraSchied2},~\cite{HanSchiedZhang1} and
\cite{HanSchiedZhang2}, the roughness exponent  was computed for certain
families of fractal functions. The standard Hurst parameter, on the other hand,
is defined via the autocorrelation of the time series and  measures the amount
of  long-range dependence in the data. For many stochastic processes, including
fractional Brownian motion, the standard Hurst parameter coincides with the
roughness exponent of the sample paths. In general, however, these two
parameters may be different; for a discussion, see   Gneiting and Schlather
\cite{Gneiting2004Hurst}, where roughness is measured in terms of fractal
dimension. Later, Bennedsen et al.~\cite{BennedsenLundePakkanen} adopted this
idea to study the bifurcation of the short- and long-term behaviour of
stochastic volatility with the goal of proposing a new class of stochastic
volatility models  that can simultaneously incorporate roughness and slowly
decaying autocorrelation. Finally, the concept of Besov regularity, introduced by Eyink~\cite{EyinkBesov} within the framework of Besov spaces, is somewhat related to our roughness exponent. The corresponding connections are explored  in  \Cref{Besov section}.

Note that if $x$ admits the roughness exponent $R$, then \eqref{eq_Hurst}
does not make any assertion on the existence of $\lim_n\<x\>_n^{(p)}$ for $p:=1/R$.
This limit always exists for $p=1$ and is equal to the total variation of $x$.
For $p>1$, however, the limit may or may not exist (see \Cref{couterexample weight quadratic}). If it does exist, it can be
interpreted as the  $p^\text{th}$ variation of the continuous function $x$ along
the sequence of dyadic partitions of $[0,1]$. This $p^\text{th}$ variation can
be used so as to prove a strictly  pathwise version of It\^o calculus, an
approach that was pioneered by F\"ollmer~\cite{FoellmerIto} for $p\le2$ and
recently extended to $p>2$ by Cont and Perkowski~\cite{ContPerkowski}.  Just as this notion of $p^{\text{th}}$ variation,  the roughness exponent of a trajectory will typically depend on the underlying partition sequence. As noted above, we have chosen here the dyadic partitions, because equally-spaced partitions are natural for many data sources and also lead to the simplest and most elegant mathematical statements. An extension to unequally spaced partitions is provided in 
\Cref{irregular section}.

There exist continuous functions that do not admit a roughness exponent;
		see \Cref{example Besov}. For this reason, we will first
		establish criteria for the existence of $R$. Let us start by defining
		\begin{equation}\label{q+q- eqn} q_*:= \sup\Big\{p>0:\liminf_{n \ua
		\infty}\<x\>_n^{(p)}  = \infty\Big\} \quad \text{and} \quad q^*:=
		\sup\Big\{p >0:\limsup_{n \ua \infty}\<x\>_n^{(p)}  = \infty\Big\}.
	\end{equation}
The proof of  \Cref{Prop_InfSup}  will show that the following alternative
formulas hold for $q_*$ and $q^*$,
\begin{equation}\label{q+q- alt eqn}
				q_*= \inf\Big\{p \ge 1:\liminf_{n \ua \infty}\<x\>_n^{(p)}  = 0\Big\} \quad \text{and} \quad q^*= \inf\Big\{p \ge 1:\limsup_{n \ua \infty}\<x\>_n^{(p)}  = 0\Big\}.
			\end{equation}

\begin{proposition}\label{Prop_InfSup} The roughness exponent of $x\in
 C[0,1]$ exists if and only if $q^* = q_*$. In this case, $R$  is given by
 $R = 1/q^* = 1/q_*$.
\end{proposition}

	\begin{proof}
	Clearly, $R$ is nonnegative by definition. Moreover, since the function $x$
	is non-constant by assumption, we have $\<x\>_n^{(p)}\to\infty$ for all $p<1$.
	Hence, we must have $R\le1$.
	
	The remaining assertion follows easily from the alternative representation
	\eqref{q+q- alt eqn} of $q_*$ and $q^*$, so it is sufficient to establish
	\eqref{q+q- alt eqn}. To prove this result for  the case $q_* < \infty$, let
	$q_0:= \inf\{p \ge 1:\liminf_{n \ua \infty}\<x\>_n^{(p)}  = 0\}$. Let $n_0$ be
	such that $\max_k|x((k+1)2^{-n})-x(k2^{-n})|\le1$ for $n\ge n_0$. Then
	$p\mapsto \<x\>_n^{(p)} $ is non-increasing for $n\ge n_0$, and so we must have
	$q_*\le q_0$. To show the converse inequality, we  show that
	$\liminf_n\<x\>_n^{(p)}  = 0$ for every $p > q_*$. To this end, we assume by
	way of contradiction that there exists $p > q_*$ such that
	$v:=\liminf_n\<x\>_n^{(p)}  > 0$. The definition of $q_*$ implies that we must
	have $v<\infty$. Next, for any $\eps>0$, there exists $n_1\in\bN$ such that
	$\max_k|x((k+1)2^{-n})-x(k2^{-n})|\le\eps$ for $n\ge n_1$. Thus, for
	$q\in(q_*,p)$ and $n\ge n_1$,
	\begin{align*}
			\<x\>_n^{(p)}   \le \max_k\big|x((k+1)2^{-n})-x(k2^{-n})\big|^{p-q}\cdot \<x\>_n^{(q)}  \le  \eps^{p-q}\<x\>_n^{(q)} .
		\end{align*}
		When sending $\eps\da0$, we see that $0<v\le
		\eps^{p-q}\liminf_n\<x\>_n^{(q)} $ for all $\eps>0$. This implies that
		$\liminf_n\<x\>_n^{(q)} =\infty$, which, in view of $q>q_*$, contradicts the
		definition of $q_*$. This establishes $q_0=q_*$. If $q_* = \infty$,
		then for each $p > 0$, we must have $\liminf_{n \ua \infty}\<x\>_n^{(p)} =
		\infty$. Thus, $\{p \ge 1:\liminf_{n \ua \infty}\<x\>_n^{(p)} = 0\} =
		\emptyset $, which establishes the first identity in \eqref{q+q- alt
		eqn} for the case $q_* = \infty$. The second identity in  \eqref{q+q-
		alt eqn} is proved in the same manner. 
	\end{proof}

Our analysis of the roughness exponent is based on the Faber--Schauder
wavelet expansion of continuous functions.  Recall that the
\textit{Faber--Schauder functions} are defined as 
	\begin{equation*}
		e_{\emptyset}(t):= t, \quad e_{0,0}(t):= (\min\{t,1-t\})^{+}, \quad e_{m,k}(t):= 2^{-m/2}e_{0,0}(2^mt-k)
	\end{equation*}
	for $t \in \bR$, $m \in \bN$ and $k \in \bZ$. It is well  known that the
	restriction of the Faber--Schauder functions to $[0,1]$ form a Schauder
	basis for $C[0,1]$. More precisely, every function $x \in C[0,1]$ can be
	uniquely represented by the following uniformly convergent series, 
	\begin{equation}\label{eq_Faber_expansion}
		x= x(0)+\left(x(1)-x(0)\right)e_{\emptyset} + \sum_{m = 0}^{\infty}\sum_{k = 0}^{2^m-1}\theta_{m,k}e_{m,k},
	\end{equation}
	where the \emph{Faber--Schauder coefficients} $\theta_{m,k}$ are given by 
	\begin{equation*}
		\theta_{m,k} = 2^{m/2}\left(2x\Big(\frac{2k+1}{2^{m+1}}\Big)-x\Big(\frac{k}{2^m}\Big)-x\Big(\frac{k+1}{2^m}\Big)\right).
	\end{equation*}
We furthermore denote for $n \in \bN$,
		\begin{equation}\label{eq_zeta_def}
		s_n:= \sqrt{\sum_{m = 0}^{n-1}\sum_{k = 0}^{2^m-1}\theta^2_{m,k} }, \quad \text{and} \quad  \wh R_n(x) := 1 - \frac{1}{n}\log_2 s_n.
	\end{equation}

	Our main thesis in this paper is that the roughness exponent of $x$
	should be equal to the limit of $\wh R_n(x) $, provided this limit exists.
	\Cref{Thm_main} will rigorously establish this result under the following
	mild condition on the Faber--Schauder coefficients of $x$. To formulate this
	and related conditions, we will say that a function $\varrho:\bR_+\to\bR_+$
	is \emph{subexponential}  if $\frac1t\log\varrho(t)\to0$ as $t\ua\infty$.

	\begin{definition}\label{Definition reverse Jensen condition} The Faber--Schauder coefficients of $x$ satisfy the
	\emph{reverse Jensen condition} if for each $p \ge 1$ there exist
	$n_p\in\bN$ and a non-decreasing subexponential function $\varrho_p: \bN
	\rightarrow [1,\infty)$ such that	
		\begin{equation}\label{reverse Jensen condition}		
			\frac1{\varrho_p(n)} s_n^p\le \frac1{2^{n-1}}\sum_{k=0}^{2^{n-1}-1}\Big(\sum_{m=0}^{n-1}2^m\theta^2_{m,\lfloor 2^{m-n+1}k\rfloor}\Big)^{p/2}\le\varrho_p(n)s_n^{p}\qquad\text{for $n\ge n_p$}.
		\end{equation}
	\end{definition}

If $p\ge2$, then Jensen's inequality implies that the left-hand inequality in
 \eqref{reverse Jensen condition} holds with $\varrho_p=1$. Likewise, the upper
 bound   holds with $\varrho_p=1$ if $p\le2$. So only the other cases are
 nontrivial, which explains our terminology \lq\lq reverse Jensen condition".
 This terminology will become even clearer in \Cref{reverse Jensen Remark},
 where an equivalent probabilistic formulation is given.  As we are going to see
 in \Cref{Theorem example FBM}, the reverse Jensen condition is satisfied for the typical
 sample paths of fractional Brownian motion with arbitrary Hurst parameter. 
 Moreover, two alternative conditions that are easier to verify, but are also
 stronger,  will  be provided in \Cref{Conditions prop}. Now we can state our
 first main result, which establishes the existence of the roughness
 exponent and its relation to the limit of the sequence
 $\wh R_n(x) $ under the assumption that the reverse Jensen condition holds. The second part of \Cref{Thm_main} states that condition \eqref{reverse Jensen condition} at $p=1/R$ is  necessary  for the convergence of $\wh R_n(x)$ to the roughness exponent $R$.

	\begin{theorem}\label{Thm_main} Let $\wh R_n$ be as in \eqref{eq_zeta_def} and suppose that the Faber--Schauder
	coefficients of $x$ satisfy the reverse Jensen condition stated in \Cref{Definition reverse Jensen condition}.	 Then the
	function $x$ admits a roughness exponent $R$ if and only if $\wh R_n(x)\to R$ as $n \ua \infty$. Conversely, suppose that $\wh R_n(x)$ converges to the roughness exponent $R$ of $x$. Then \eqref{reverse Jensen condition} holds at $p = 1/R$.
\end{theorem}

\Cref{Thm_main} will be proved in the subsequent \Cref{martingale section}. It
can be regarded as a model-free consistency result for the estimator $\wh R_n$ of the roughness exponent of the continuous function $x$. It has the following two aspects.

First, \Cref{Thm_main} is  model-free  in the sense that no assumptions are made on $x$
except the reverse Jensen condition. In particular, we do not assume that $x$ is
a realization of a stochastic process with prescribed dynamics, even though this is clearly the typical application we have in mind. This line of thought is similar in spirit to the model-free approach to stochastic analysis, which was pioneered by F\"ollmer~\cite{FoellmerIto} and Lyons~\cite{Lyons98}. In this approach, one derives results for deterministic trajectories satisfying 
 certain sample path properties, thus avoiding other implicit assumptions  (e.g., a L\'evy-type modulus of continuity or a law of the iterated logarithm) usually imposed through fixing a probabilistic model. The results can then be applied to  all probabilistic models whose trajectories satisfy the assumed sample path properties almost surely. In  our analysis of fractional Brownian motion with drift (\Cref{Gaussian examples section}) we follow precisely this recipe. We show that almost all of the  sample paths satisfy the reverse Jensen condition and we identify the Hurst parameter $H$ as the common roughness exponent of the trajectories. We also provide almost sure convergence rates for the convergence of $\wh R_n$ to $H$. 
 
 Second,  based on the
work by Gladyshev~\cite{Gladyshev}, it has long been known that $\wh R_n$, which is  sometimes called the \emph{Gladyshev estimator}, estimates the Hurst parameter for fractional Brownian motion and other Gaussian processes; see, e.g.,  \cite{KubiliusMishura}  
for an overview.  \Cref{Thm_main} now clarifies   that  $\wh R_n$ is actually an estimator   for the roughness exponent; it  estimates the classical Hurst parameter $H$ 
only if $H$  coincides with the roughness exponent. Our theorem thus demonstrates that the Gladyshev estimator has in fact a much wider scope of possible applications, provided that it is correctly understood as an estimator of the roughness exponent and not of the Hurst parameter. The same is true for certain variants of the Gladyshev estimator such as the scale-invariant modifications discussed in \Cref{stat section}.

The statement of \Cref{Thm_main} may become false if the reverse Jensen condition is not satisfied; see \Cref{remark Jensen counter} for a counterexample. 
However, as shown by the next result, an exception occurs  for functions $x$ that display diffusive behavior in the sense that they admit the roughness exponent $R=1/2$. This includes in particular the typical sample paths of any continuous semimartingale with nontrivial martingale component.

\begin{corollary}\label{diffusive cor} Let $x\in C[0,1]$ be any function with roughness exponent $1/2$. Then $\lim_n\wh R_n(x)=1/2$. 
\end{corollary}

The preceding corollary will follow immediately from  \Cref{Prop_Bias}. 
 The following remark outlines what can be shown if the reverse Jensen condition holds but $\wh R_n(x)$ does not converge to a limit.

		\begin{remark}\label{xipm remark}
		 Suppose that the reverse Jensen condition holds and let us denote 
			\begin{equation}\label{eq xi pm}
				 r^+:= \limsup_{n \ua \infty} \wh R_n(x) \quad \text{and} \quad  r^-:= \liminf_{n \ua \infty} \wh R_n(x) .
			\end{equation}
			One can show as in the proof of \Cref{Thm_main} that
			\begin{equation*}
				\limsup_{n \ua \infty}\<x\>_n^{(p)}  = \begin{cases}
					0 &\text{if } p >  1/r^- ,\\
					\infty &\text{if } p <  1/r^- ,\\
				\end{cases} \qquad\qquad				\liminf_{n \ua \infty}\<x\>_n^{(p)}  = \begin{cases}
					0 &\text{if } p >  1/r^+ ,\\
					\infty &\text{if } p <  1/r^+ .\\
				\end{cases}
			\end{equation*}
		 To wit, under the reverse Jensen condition, we have $q^\ast = 1/r^-$ and $q_\ast = 1/r^+$. By \Cref{Prop_InfSup}, the roughness
		exponent is well defined if and only if $r^+ = r^-$, and this
		special case is discussed in \Cref{Thm_main}; on the other hand, if
		$r^+ > r^-$, the roughness exponent does not exist, and
		the sequence $(\<x\>_n^{(p)})_{n \in \bN_0} $ will fluctuate between zero
		and  infinity for any $p \in  (1/r^+,1/r^-)$. 
	\end{remark}
	
	\begin{remark}\label{Takagi class rem} Suppose $x\in C[0,1]$ is such that its Faber--Schauder coefficients  $\theta_{m,k}$ are the same for all~$k$. Then the
	reverse Jensen condition is trivially satisfied. These functions form the so-called Takagi class~\cite{HataYamaguti}, and, beginning with the subsequent Examples \ref{TL ex} and \ref{H=0 example}, several of our examples and counterexamples will belong to this class. The two Examples \ref{TL ex} and \ref{H=0 example} will in particular present for every $R\in[0,1]$ a function $x$ that satisfies the reverse Jensen condition and has roughness exponent $R$. 	\end{remark}

\begin{example}\label{TL ex} For $\alpha\in(-\sqrt2, \sqrt2)$, consider the
	function with Faber--Schauder coefficients $\theta_{m,k}=\alpha^m$.
	These functions are often called the Takagi--Landsberg functions. For $|\alpha|>1/\sqrt2$, it follows from~\cite[Theorem 2.1]{MishuraSchied2} that their roughness exponent is given by $R=\frac12-\log_2|\alpha|$. Our \Cref{Thm_main}  provides a very short proof of this fact: since
	$$s_n^2=\sum_{m=0}^{n-1}\sum_{k=0}^{2^m- 1}\alpha^{2m}=\frac{(2\alpha^2)^n-1}{2\alpha^2-1},
	$$
	we have
	$\frac1n\log_2s_n\to\frac12\log_2(2\alpha^2)=\frac12+\log_2|\alpha|$,
	which yields the claim. For $|\alpha|<1/\sqrt2$, we get in the same way that $s_n^2\to (1-2\alpha^2)^{-1}$. It will be shown in \Cref{BV prop} (a) that $x$ must hence be of bounded variation. Finally, in the case $|\alpha|=1/\sqrt2$, we have $s_n^2=n$, which also leads to $R=1$. However, in this case, $x$ is the classical Takagi function, which is nowhere differentiable and hence not of bounded variation;  see Remark 2.2 (i) in~\cite{SchiedZZhang} for  details.
\end{example}

On the other extreme end of the spectrum are functions with roughness
exponent $R=0$. Such a function is constructed in the following example.

\begin{example}\label{H=0 example} Consider the function with Faber--Schauder coefficients $\theta_{m,k} = 2^{m/2}(m+1)^{-2}$.
	One easily checks that this choice gives rise to a
	well-defined continuous function.  As in~\cite[Proposition
	3.1]{HanSchiedZhang2}, one can show that $\<x\>_n^{(p)}\to\infty$ for all
	$p\ge1$. It follows that $q_*$ and $q^*$, as defined in \eqref{q+q- eqn},
	are both infinite, and so \Cref{Prop_InfSup} yields $R=0$. 
\end{example}

In the following \Cref{martingale section}, we present the proof of
\Cref{Thm_main}. It is based on a martingale that is hidden in the
Faber--Schauder expansion of a continuous function $x$.  

\section{A martingale hidden in the Faber--Schauder expansion of a continuous
function}\label{martingale section}

In this section, we use martingale techniques so as to develop the probabilistic tools for our analysis of the roughness exponent and we will prove  \Cref{Thm_main}. The idea of using a probabilistic approach to the $p^{\text{\rm th}}$ variation of functions goes back to~\cite{MishuraSchied2,SchiedZZhang}, where it was used for for certain fractal functions. For a general function $x\in C[0,1]$, however, the tools from~\cite{MishuraSchied2,SchiedZZhang} are no longer sufficient; see \Cref{Remark independent}.

Let $(\Omega,\cF, \bP)$ be a probability space supporting an i.i.d.~sequence $(U_n)_{n \in \bN}$ of $\{0,1\}$-valued random variables with symmetric Bernoulli distribution.
	 Furthermore, we define the stochastic processes 
	\begin{equation}\label{eq_Sm_Def}
	V_0:=0,\quad	V_m :=\sum_{k = 1}^{m}2^{-k}U_k \quad \text{and} \quad S_m :=2^m\theta^2_{m,2^mV_m}, \quad m \in \bN_0.
	\end{equation}
	 Note that $V_m$ is uniformly distributed on $\{k2^{-m}:k=0,\dots, 2^m-1\}$.

	We next define the filtration $\cF_0 := \{\emptyset,\Omega\}$ and
	$\cF_n:=\sigma(U_1,\cdots,U_n)$ for $n \in \bN$. Since $U_1, U_2, \cdots,
	U_n$ can be uniquely recovered from $V_n$, we also have $\cF_n =
	\sigma(V_n)$. Let $Y_n:= 2^{n/2}\theta_{n,2^nV_n}\cdot(1-2U_{n+1})$. As $(1-2U_{m})$ defines an i.i.d.~sequence of centered
	random variables, and $\theta_{m,2^mV_m}$ is $\cF_m$-measurable, the
	sequence
	   $$X_n := \sum_{m =
	   	0}^{n-1}Y_m = \sum_{m =
	    0}^{n-1}2^{m/2}\theta_{m,2^mV_m}\cdot(1-2U_{m+1}),\qquad n\in\bN,$$ is
	    an $(\cF_n)$-martingale.  The link between this martingale and the
	    quantities  $\<x\>_n^{(p)}$ is established in the following proposition,
	    which  provides the key to the entire analysis in this paper. 
	
	\begin{proposition}\label{Proposition_Burk} Suppose that  $x(0) = x(1) = 0$.
	Then, for  $p \ge 1$ and $n\in\bN$, 
		\begin{equation}\label{martingale rep vor Vn}
		\<x\>_n^{(p)}=2^{n(1-p)}\bE\big[|X_n|^p\big].
		\end{equation}
	Moreover, there exist  constants $0 < A_p \le B_p < \infty$ depending only
				on $p$  but not on $x$, such that for all $n\in\bN$,
		\begin{equation}\label{eq_burkholder}
			A_p2^{n(1-p)}\bE\Big[\Big(\sum_{m = 0}^{n-1}S_m\Big)^{p/2}\Big] \le \<x\>_n^{(p)}  \le B_p2^{n(1-p)}\bE\Big[\Big(\sum_{m = 0}^{n-1}S_m\Big)^{p/2}\Big].
		\end{equation}
		Furthermore, in the special case $p=2$, we have
				\begin{equation}\label{QV eq}
		\<x\>_n^{(2)}  =2^{-n}s_n^2.
		\end{equation}
	\end{proposition}

The formula \eqref{QV eq} was first obtained by Gantert~\cite{Gantert} via a different proof. It can be used to simplify the computation of $s_n$ and $\wh R_n(x)$.

\begin{remark}\label{Remark independent}
Suppose that $x$ belongs to the Takagi class \cite{HataYamaguti}, and $\theta_{m,k} = \alpha_m$ for a sequence of real numbers $(\alpha_m)$ for which the series $\sum_{m = 0}^{\infty}2^{m/2}\alpha_m$ converges absolutely. The roughness exponent of functions in the Takagi class was studied in \cite{MishuraSchied2, SchiedZZhang, HanSchiedZhang2}. In this case, $Y_n = 2^{n/2}\alpha_n(1 - 2U_{n+1})$ defines a sequence of independent centered random variables, and the sequence $(X_n)$ remains an $(\cF_n)$-martingale. The independence of $(Y_n)$ allows the application of Khintchine's inequality, which directly gives
\begin{equation*}
A_p2^{n(1-p/2)}\left(\<x\>^{(2)}_n\right)^{p/2} = A_p2^{n(1-p)}(s_n^2)^{p/2} \le	\<x\>^{(p)}_n  \le B_p2^{n(1-p)}(s_n^2)^{p/2} = B_p2^{n(1-p/2)}\left(\<x\>^{(2)}_n\right)^{p/2},
\end{equation*}
where $A_p$ and $B_p$ are constants with $0 < A_p \le B_p < \infty$, depending only on $p$ but not on $x$. The analysis in \cite{SchiedZZhang, HanSchiedZhang2} then proceeds from the above inequality. However, for a general function $x \in C[0,1]$, the random variables $(Y_n)$ may no longer be independent. Thus, the techniques in \cite{MishuraSchied2, SchiedZZhang, HanSchiedZhang2} fail to apply to the scope of this paper.
\end{remark}

	 \begin{remark}\label{reverse Jensen Remark} In this remark, we collect several alternative formulations of the reverse Jensen condition. Note that 
	\begin{equation}\label{eq S_m power}
		\bE\Big[\sum_{m=0}^{n-1}S_m\Big]=s_n^2\qquad\text{and}\qquad \bE\bigg[\bigg(\sum_{m=0}^{n-1}S_m\bigg)^{p/2}\bigg]=\frac1{2^{n-1}}\sum_{k=0}^{2^{n-1}-1}\Big(\sum_{m=0}^{n-1}2^m\theta^2_{m,\lfloor 2^{m-n+1}k\rfloor}\Big)^{p/2}.
	\end{equation}
	We can thus rephrase the reverse Jensen condition as follows: there exist
	$n_p\in\bN$ and a non-decreasing subexponential function $\varrho_p: \bN_0
	\rightarrow [1,\infty)$ such that	
	\begin{equation}\label{Thm_main cond}
		\frac1{\varrho_p(n)} \bE\Big[\sum_{m=0}^{n-1}S_m\Big]^{p/2}\le \bE\bigg[\bigg(\sum_{m=0}^{n-1}S_m\bigg)^{p/2}\bigg]\le\varrho_p(n) \bE\Big[\sum_{m=0}^{n-1}S_m\Big]^{p/2}\qquad\text{for $n\ge n_p$}.
	\end{equation}
	
	Moreover, \Cref{Proposition_Burk} shows that the reverse Jensen condition can  be reformulated by means of $p^{\rm th}$ variations: for each $p \ge 1$, there exist $n_p \in \bN$ and a non-decreasing subexponential function $\varrho_p: \bN_0 \rightarrow [1,\infty)$ such that 
		\begin{equation}\label{eq Reverse Jensen variation}
			\frac{1}{\varrho_p(n)}2^{n(1-p/2)}\left(\<x\>^{(2)}_n\right)^{p/2} \le \<x\>^{(p)}_n \le \varrho_p(n)2^{n(1-p/2)}\left(\<x\>^{(2)}_n\right)^{p/2} \qquad \text{for}~ n \ge n_p.
		\end{equation}
Comparing  \eqref{eq Reverse Jensen variation} with the same inequality for a different value $q\ge1$ gives the following more general equivalent statement of the reverse Jensen condition: for every $1 \le p,q < \infty$, there exists $n_{p,q} \in \bN$ and a non-decreasing subexponential function $\varrho_{p,q}: \bN_0 \rightarrow [1,\infty)$ such that 
		\begin{equation}\label{eq Reverse Jensen variation 2}
			\frac{1}{\varrho_{p,q}(n)}2^{n(1-p/q)}\left(\<x\>^{(q)}_n\right)^{p/q} \le \<x\>^{(p)}_n \le \varrho_{p,q}(n)2^{n(1-p/q)}\left(\<x\>^{(q)}_n\right)^{p/q} \qquad \text{for}~ n \ge n_{p,q}.
		\end{equation}
		If $p \ge q$, then the left-hand inequality in \eqref{eq Reverse Jensen variation 2} holds with $\varrho_{p,q} = 1$. Likewise, if $p \le q$, then the right-hand inequality in \eqref{eq Reverse Jensen variation 2} holds with $\varrho_{p,q} = 1$, which can also be understood as the well-known inequality for generalized means.
	\end{remark}

For the proof of \Cref{Proposition_Burk}, we need the following lemma. 
	
	\begin{lemma}\label{Lemma_Binary} For $n \in \bN$ and $m \le n-1$, if $t$ is
		a  dyadic rational number in $ [0,1)$ with binary expansion $ t =
		\sum_{j = 1}^{n}2^{-j}a_j
		$
				for $a_j \in \{0,1\}$, we have 
		$
			e_{m,k}(t+2^{-n}) - e_{m,k}(t) = 2^{m/2-n}\left(1-2a_{m+1}\right)\Ind{\{ \lfloor2^mt\rfloor=k\}}.
		$
		\end{lemma}
		
		\begin{proof}
			For $m \le n-1$ and $\psi := \Ind{[0,1/2)}-\Ind{[1/2,1)}$, we have 
			\begin{equation}\label{eq_Binary_diff}
				\begin{split}
					e_{m,k}(t+2^{-n}) - e_{m,k}(t) &= 2^{-m/2}\left(e_{0,0}(2^mt+2^{m-n}-k)-e_{0,0}(2^mt-k)\right)=2^{m/2-n}\psi(2^mt-k),
				\end{split}    	
			\end{equation}
			Furthermore, it follows that
			\begin{equation*}
				\begin{split}
					\psi(2^mt-k) &= \psi\Big(\sum_{j = 1}^{n}2^{m-j}a_j-k\Big) = \psi\left( \sum_{j = m+1}^{n}2^{m-j}a_j-\Big(k-2^m\sum_{j = 1}^{m}2^{-j}a_j\Big)\right)\\ &=\psi\left(\frac12a_{m+1} + \sum_{j = m+2}^{n}2^{m-j}a_j-(k-\lfloor2^mt\rfloor)\right).
				\end{split}
			\end{equation*}
			As $k-\lfloor2^mt\rfloor \in \bZ$ and $\sum_{j = m+1}^{n}2^{m-j}a_j
			\in [0,1)$, we have $\psi(2^mt-k) \neq 0$ if and only if
			$\lfloor2^mt\rfloor = k$. Furthermore, since $\sum_{j =
			m+2}^{n}2^{m-j}a_j \in [0,1/2)$ and $a_{m+1}/2 \in \{0,1/2\}$, we
			must have 
			$$\psi\left(\frac12a_{m+1} + \sum_{j = m+2}^{n}2^{m-j}a_j\right) =
			\psi\Big(\frac12a_{m+1}\Big) = 1-2a_{m+1}.
			$$
			 Hence, 
			$
				\psi(2^mt-k) = \left(1-2a_{m+1}\right)\Ind{\{k = \lfloor2^mt\rfloor\}}
			$.		Substituting this identity back into \eqref{eq_Binary_diff} completes the proof.
		\end{proof}

	\begin{proof}[Proof of \Cref{Proposition_Burk}]
		For $n \in \bN_0$, let the $n^{\text{\rm th}}$ truncation $x_n$ of $x$ be given by
		\begin{equation}\label{truncation eq}
			x_n(t) := \sum_{m = 0}^{n-1}\sum_{k = 0}^{2^m-1}\theta_{m,k}e_{m,k}(t).
		\end{equation} 
		Since the Faber--Schauder functions $e_{m,k}$ vanish on $\{k2^{-n}:k=0,\dots, 2^n\}$ for $m \ge n$, we have $x(k2^{-n}) = x_n(k2^{-n})$ for $k=0,\dots, 2^n$. Using  the fact that $V_n$ is uniformly distributed on $\{k2^{-n}:k=0,\dots, 2^n-1\}$ and \Cref{Lemma_Binary}, we get
		\begin{equation}\label{eq_long_expectation}
			\begin{split}
				\<x\>_n^{(p)}  &= 2^n\bE\left[|x_n(V_n+2^{-n}) - x_n(V_n)|^p\right]
				\\
				& = 2^n\bE\left[\left|\sum_{m = 0}^{n-1}\sum_{k = 0}^{2^m-1}\theta_{m,k}\big(e_{m,k}(V_n+2^{-n}) - e_{m,k}(V_n)\big)\right|^p\right]\\&= 2^n\bE\left[\left|\sum_{m = 0}^{n-1}2^{m/2-n}\sum_{k = 0}^{2^m-1}\theta_{m,k}\left(1-2U_{m+1}\right)\Ind{\{k = 2^mV_m\}}\right|^p\right]\\&= 2^{n(1-p)}\bE\left[ \left|\sum_{m = 0}^{n-1}2^{m/2}\theta_{m,2^mV_m}(1-2U_{m+1})\right|^p \right]= 2^{n(1-p)}\bE\big[|X_n|^p\big],
			\end{split}
		\end{equation}
which yields \eqref{martingale rep vor Vn}.  Therefore, the Burkholder inequality implies the existence of constants $0 < A_p \le B_p < \infty$ depending only on $p$ such that
		\begin{equation}\label{eq_Inq}
			A_p2^{n(1-p)}\bE\Big[\Big(\sum_{m = 0}^{n-1}Y_m^2\Big)^{p/2}\Big] \le \<x\>_n^{(p)}  \le B_p2^{n(1-p)}\bE\Big[\Big(\sum_{m = 0}^{n-1}Y_m^2\Big)^{p/2}\Big].
		\end{equation}
		Moreover, as $|1-2U_{m+1}| = 1$, we have
		\begin{equation}\label{Burk p=2 eq}
			Y_m^2 = \left(2^{m/2}\theta_{m,2^mV_m}\right)^2\left(1-2U_{m+1}\right)^2 = 2^m\theta^2_{m,2^mV_m} = S_m.
		\end{equation}
	This completes the proof of \eqref{eq_burkholder}  for arbitrary $p>1$.
In the case $p=2$, the martingale differences $(Y_m)$ clearly satisfy
$\bE[\sum_{m = 0}^{n-1}Y_m^2]=\bE[(\sum_{m = 0}^{n-1}Y_m)^{2}]$,	
which yields \eqref{Burk p=2 eq} by way of \eqref{eq_long_expectation}, \eqref{eq_Inq}, and \eqref{Burk p=2 eq}. This completes the proof. \end{proof}

Now we move toward the proof of \Cref{Thm_main}. The following auxiliary lemma will be needed in that proof.

\begin{lemma}\label{lemma_Range}
	 If  the finite limit $\lim_n \wh R_n(x)$ exists, then it belongs to $ [0,1]$. 
\end{lemma}

\begin{proof}
Let  $\xi_n:= 1 - \wh R_n(x) = \frac{1}{n}\log_2 s_n$. It suffices to show that  $\xi:=\lim_n\xi_n=1- \lim_{n}\wh R_n(x)$ belongs to $[0,1]$. First, let us assume by way of contradiction that  $\xi<0$. Thus, there exists $n_1 \in \bN$ such that for $n \ge n_1$, we have $\xi_n \le -n\xi  /2$ and $s_n \le 2^{-n\xi  /2}$. However, this contradicts to the fact $(s_n)_{n \in \bN}$ is a  non-negative and nondecreasing sequence. Therefore, we must have $\xi  \ge 0$.

	Next, we assume by way of contradiction that  $\xi\in(1,\infty) $. For each $n \in \bN$, we take $\lambda(n):= 2^{-2n\xi }s_n^2$. Then $\lambda$ is a subexponential function, and 
	$
		\sum_{m = 0}^{n-1}\sum_{k = 0}^{2^m-1}\theta_{m,k}^2 = \lambda(n)2^{2n\xi }.
	$	It then follows that
	$		\sum_{k = 0}^{2^n-1}\theta^2_{n,k} = \lambda(n+1)2^{2\xi (n+1)} -  \lambda(n)2^{2n\xi },
	$
	and
	$
		\sup_{0 \le k \le 2^n-1}|\theta_{n,k}| \ge 2^{(\xi-1/2)n}\sqrt{\lambda(n+1)2^{2\xi }-\lambda(n)}.
	$
	As in \eqref{truncation eq}, let $x_n$
		 denote the $n^\text{th}$ truncation of $x$, then
	\begin{equation*}
		\max_{0\le t\le1}|x_{n+1}-x_n| = 2^{-n/2}\sup_{0 \le k \le 2^n-1}|\theta_{n,k}| \ge  2^{(\xi-1) n}\sqrt{\lambda(n+1)2^{2\xi}-\lambda(n)},
	\end{equation*}
	Since $\lambda$ is a subexponential function, the rightmost term tends to infinity as $n \ua \infty$. Thus, the series of truncated functions $(x_n)_{n \in \bN}$ will not converge. This contradicts the uniform convergence of $x_n$  to~$x$. Thus, we must have $\xi \le 1$, and this completes the proof.\end{proof}

	\begin{proof}[Proof of \Cref{Thm_main}] By the argument used in the proof of Lemma 3.1 of~\cite{MishuraSchied2}, we may assume without loss of generality that $x(0)=x(1)=0$.
		We start by proving  the \lq\lq only if"  direction in the first part of the assertion and assume that the function $x$ admits the roughness exponent $R$. For simplicity, we use the short-hand notion $r_n:= \wh R_n(x)$ in this and following proofs of this paper. We must prove that $ r_n \to R$. For any $p > 1/R \ge 1$, \Cref{Proposition_Burk} and condition \eqref{Thm_main cond} yield that there exists $A_p > 0$ such that 
		\begin{equation}\label{eq p-2 var}
			A_p2^{n(1-pr_n-\log_2\varrho_p(n)/n)} = A_p\varrho_p^{-1}(n)2^{n(1-p)}s^p_n = A_p\varrho_p^{-1}(n)2^{n(1-p)}\bE\Big[\sum_{m = 0}^{n-1}S_m\Big]^{p/2}\le \<x\>_n^{(p)} .
		\end{equation}
		Since $\limsup_n\<x\>_n^{(p)}  = 0$, we must thus have  $\liminf_{n}r_n > 1/p$. Sending $p\da1/R$ now gives $\liminf_n r_n \ge R$. In the same way, one proves that  $ \limsup_n r_n \le R$.
		
		To prove the \lq\lq if"  direction, let us assume that the finite limit  $\lim_n r_n = r$ exists. Then $r \in [0,1]$ by Lemma \ref{lemma_Range}, and this yields
		 \begin{equation*}
			\lim_{n \ua \infty}n\Big(1-pr_n\pm\frac1n\log_2\varrho_p(n)\Big) = \begin{cases}
				+\infty &\quad \text{if} \quad p < 1/r,\\
				-\infty &\quad \text{if} \quad p > 1/r.
			\end{cases}
		\end{equation*}
		Applying condition \eqref{Thm_main cond} to \eqref{eq_burkholder} and taking $n \ua \infty$ implies that $\lim_n \<x\>_n^{(p)}  = \infty$ for $p < 1/r$ and $\lim_n \<x\>_n^{(p)}  = 0$ for $p > 1/r$; Therefore, we must have $R = r$. This completes the proof.
		
		Let us now prove the converse assertion by verifying its contrapositive statement. To this end, we first assume that condition \eqref{reverse Jensen condition} does not hold with $p = 1/r > 2$, which implies that there exist $\varepsilon > 0$ and a subsequence $(n_k)_{k \in \bN}$ such that 
		\begin{equation*}
			\<x\>^{(p)}_{n_k} \ge 2^{n_k(1-p/2+\varepsilon)}\left(\<x\>^{(2)}_{n_k}\right)^{p/2} = 2^{n_k(1-p+\varepsilon)}\left(s_{n_k}^2\right)^{p/2} = 2^{n_k(1-pr_{n_k}+\varepsilon)}.
		\end{equation*}
		As $r_n \rightarrow r = 1/p$, then $pr_{n_k} \le 1 + \varepsilon/2$ for sufficiently large $k$. Thus, there exists $k_0 \in \bN$ such that for $k \ge k_0$, we get $\<x\>^{(p)}_{n_k} \ge 2^{\varepsilon n_k/2}$. Let $\delta > 0$ be such that $\delta/(p+\delta) < \varepsilon/4$, then it follows from \eqref{eq Reverse Jensen variation 2} that 
		\begin{equation*}
			\<x\>^{(p+\delta)}_{n_k} \ge \left(2^{n_k\left(\frac{p}{p+\delta}-1\right)}\<x\>^{(p)}_{n_k}\right)^{\frac{p+\delta}{p}} \ge 2^{\frac{p+\delta}{p}\left(\frac{\varepsilon}{2}-\frac{\delta}{p+\delta}\right)n_k} \ge 2^{\frac{(p+\delta)\varepsilon}{4p} n_k }.
		\end{equation*}
		The above inequality implies $\limsup_{n} \<x\>^{(p+\delta)}_{n} = \infty$, and \Cref{Prop_InfSup} indicates that $x$ must admit the roughness exponent $R \le (p+\delta)^{-1} < r$. The proof for the case $p = 1/r < 2$ is analogous, and condition \eqref{reverse Jensen condition} holds trivially at $p = 2$. This completes the proof.
	\end{proof}

The following proposition provides two alternative conditions on the Faber--Schauder coefficients of an arbitrary continuous function~$x$ that may be easier to verify than the reverse Jensen condition but are also slightly stronger. 	
	
	\begin{proposition}\label{Conditions prop}
	Consider the following two conditions on the Faber--Schauder coefficients of $x$.
	\begin{enumerate}
			\item \label{Condi_a} There exists an non-decreasing subexponential function $\gamma_1$ such that for sufficiently large $m$,
			\begin{equation*}
				{\max_{k=0,\dots, 2^m-1}|\theta_{m,k}|}\le \gamma_1(m){\min_{k=0,\dots, 2^m-1}|\theta_{m,k}|} .
			\end{equation*}
			\item \label{Condi_b} There exist $\nu \in \bN_0$ and an non-decreasing subexponential function $\gamma_2 : \bN \rightarrow \bR$ such that for sufficiently large $m \ge \nu$, 
			\begin{equation*}
				\max_{k=0,\dots,2^{m-\nu}-1}\sum_{j = 2^{\nu}k}^{2^{\nu}(k+1)-1}\theta^2_{m,j} \le \gamma_2 (m) \min_{k=0,\dots,2^{m-\nu}-1}\sum_{j = 2^{\nu}k}^{2^{\nu}(k+1)-1}\theta^2_{m,j}.
			\end{equation*}
						\end{enumerate}
Then  \ref{Condi_a} implies \ref{Condi_b}, and \ref{Condi_b} implies the reverse Jensen condition \eqref{reverse Jensen condition}.	\end{proposition}

\begin{proof}
		\ref{Condi_a} $\Rightarrow$ \ref{Condi_b}: If condition \ref{Condi_a} holds, then \ref{Condi_b} holds with $\nu = 0$ and $\gamma_2  = \gamma_1^2$.
		
		\ref{Condi_b} $\Rightarrow$ \eqref{reverse Jensen condition}: For the ease of notation, we  prove this implication only for the case $\nu = 1$; the case $\nu \ge 2$ can be proved analogously. Recall that the upper bound in the reverse Jensen condition holds with $\varrho_p\equiv1$  if $q:=p/2\le 1$. So it is sufficient to deduce this upper bound for $q>1$. We will use the probabilistic formalism introduced in \eqref{eq_Sm_Def}. Note first that 
		$$\frac1{2^{n-1}} \sum_{k=0}^{2^{n-1}-1}\Big(\sum_{m=0}^{n-1}2^m\theta^2_{m,\lfloor 2^{m-n}k\rfloor}\Big)^{p/2}=\bE\bigg[\bigg(\sum_{m=0}^{n-1}S_m\bigg)^{q}\bigg].
		$$
		Let $a_0,\dots, a_{n-1}\ge0$ and note  that  
		$(\sum_ia_i)^q\le n^{q-1}\sum_ia_i^q$  by Jensen's inequality. Hence,
				\begin{align*}
			\bE\bigg[\bigg(\sum_{m=0}^{n-1}S_m\bigg)^{q}\bigg]
			&\le n^{q-1}\bE\bigg[\sum_{m=0}^{n-1}S_m^{q}\bigg]= n^{q-1} \sum_{m = 0}^{n-1}2^{-m}\sum_{k = 0}^{2^m-1}2^{mq}|\theta_{m,k}|^{2q}
\\
&\le  2^{q-1}n^{q-1}\bigg(|\theta_{0,0}|^{2q}+ \sum_{m = 1}^{n-1}2^{-m}\sum_{k = 0}^{2^{m-1}-1}\Big(2^m\theta^2_{m,2k}+2^m\theta^2_{m,2k+1}\Big)^{q}\bigg)
\\&\le  2^{q-1}n^{q-1} \bigg(|\theta_{0,0}|^{2q}+  \sum_{m =1}^{n-1}\frac{1}{2}\Big(\max_k(2^m\theta^2_{m,2k}+2^m\theta^2_{m,2k+1})\Big)^{q}\bigg)
\\&\le  2^{q-1}n^{q-1}  (\gamma_2(n)) ^{q} \bigg(|\theta_{0,0}|^{2q}+  \sum_{m =1}^{n-1}\frac{1}{2}\Big(\min_k(2^m\theta^2_{m,2k}+2^m\theta^2_{m,2k+1})\Big)^{q}\bigg)\\
&\le  2^{q-1}n^{q-1}  (\gamma_2(n)) ^{q}\sum_{m = 0}^{n-1}\bE[S_m]^{q}\le  2^{q-1}n^{q-1} (\gamma_2(n)) ^{q}\bigg(\sum_{m=0}^{n-1}\bE[S_m]\bigg)^{q}.
		\end{align*}
Selecting $\varrho_p(n):= 2^{p/2-1}n^{p/2-1} (\gamma_2(n)) ^{p/2}$ and applying \eqref{eq S_m power} on the above long inequality then yields upper bound in \eqref{reverse Jensen condition}.		
For $0<q<1$, an analogous reasoning as above yields that		\begin{align*}
		\bE\bigg[\bigg(\sum_{m=0}^{n-1}S_m\bigg)^{q}\bigg] &\ge
		n^{q-1}2^{q-1}(\gamma_2(n))^{-q}\sum_{m =
		0}^{n-1}\bE[S_m]^q\ge\frac1{\varrho_p(n)}s_n^p,		\end{align*} where
		$\varrho_p(n):=n^{1-p/2}2^{1-p/2}(\gamma_2(n))^{p/2}$. This establishes
		the lower bound in the reverse Jensen condition, and we complete our
		proof. \end{proof}
		
		\section{The connection with weighted quadratic variation}\label{weighted QV section}
	
The Gladyshev estimator has its origins in Gladyshev's theorem~\cite{Gladyshev}, which is a limit theorem for the weighted quadratic variation of fractional Brownian motion. In this section, we discuss the relations between $\wh R_n(x)$ and the weighted quadratic variation of an arbitrary function $x\in C[0,1]$.  We start with the following proposition, which states general implications. Its part (b) introduces the estimator $\wh R^*_n$, whose simplified form can sometimes make the analysis easier.

\begin{proposition}\label{weighted QV prop}
Let $x \in C[0,1]$ have the Faber--Schauder coefficients $\theta_{m,k}$.
\begin{enumerate}
\item\label{weighted QV prop part a}
 The following conditions are equivalent. 
	\begin{enumerate}[{\rm (i)}]		\item\label{weighted QV prop part i} There exists $r_1\in(0,1)$ such that the weighted quadratic variation $2^{n(2r_1-1)}\<x\>^{(2)}_n$ converges to the limit $\ell_1\in(0,\infty)$.
				\item\label{weighted QV prop part ii} There exists $r_2\in(0,1)$ such that  $2^{n(2r_2-2)}s_n^2$ converges to the  limit $\ell_2\in(0,\infty)$.
						\item There exists $r_3\in(0,1)$ such that $2^{n(2r_3-2)}(s_{n+1}^2-s_n^2)$ converges to the  limit $\ell_3\in(0,\infty)$.
		\end{enumerate}
	In this case, we have $r:=r_1=r_2=r_3$ and $\ell_1= \ell_2 = (2^{2-2r}-1)^{-1}\ell_3$.
	\item\label{weighted QV prop part b} The equivalent conditions in part \ref{weighted QV prop part a}
imply that $\lim_n\wh R^*_n(x)=r$, where
$$\wh R^*_n(x):=1-\frac1n\log_2\sqrt{\sum_{k = 0}^{2^n-1}\theta_{n,k}^2}.
$$
\item\label{weighted QV prop part c} If $\wh R^*_n(x)$ converges to the finite limit $r$, then also $\lim_n\wh R_n(x)=r$.
	\end{enumerate}
\end{proposition}

\begin{proof}[Proof of \Cref{weighted QV prop}] \ref{weighted QV prop part a}: To see the equivalence of (i) and (ii) as well as $r_1=r_2$ and $\ell_1=\ell_2$, one argues first as in the proof of~\cite[Proposition 2.1]{MishuraSchied} that one can assume without loss of generality that $x(0)=x(1)=0$. Then the claim follows from \eqref{QV eq}. To see that (ii) implies (iii) with $r_3=r_2$ and $ \ell_2 = (2^{2-2r_2}-1)^{-1}\ell_3$, we let $b_n = 2^{n(2-2r_2)}$. Then $s_n^2/b_n \rightarrow \ell_2$ by assumption.  It follows from~\cite[Lemma 3.1]{MishuraSchied} that 
	\begin{equation}\label{Stolz Cesaro eq}
	\dfrac{2^{(2r_2-2)n}}{2^{2-2r_2}-1}\big(s_{n+1}^2-s_n^2\big)=\frac{s_{n+1}^2 - s_n^2}{b_{n+1} - b_n}  \longrightarrow \ell_2 \quad \text{as} \quad  n \ua \infty.
	\end{equation}
For the proof that (iii) implies (ii) with $ \ell_2 = (2^{2-2r_2}-1)^{-1}\ell_3$ and $r_2=r_3$, it suffices to note that the convergence \eqref{Stolz Cesaro eq}
 implies that $s_n^2/b_n\to\ell_2$ by means of the Stolz--Cesaro theorem in the form of~\cite[Theorem 1.22]{Muresan}.

\ref{weighted QV prop part b}: Using (iii) yields that $r-\wh R_n^*(x)= (r-1)+\frac{1}{2n}\log_2(s^2_{n+1}-s^2_n) = \frac1{2n}\log_2(2^{n(2r-2)}(s_{n+1}^2-s_n^2))\to0$.
  
 \ref{weighted QV prop part c}:  Let $\eps\in(0,2-2r)$ be given. By part \ref{weighted QV prop part b}, there is a constant $c>0$ such that $c^{-1}2^{(2-2r-\eps)n}\le(s_{n+1}^2-s_n^2)\le c2^{(2-2r+\eps)n}$ for all $n\in\bN$. It follows that 
 $$\frac{2^{(2-2r-\eps)n}-1}{c(2^{2-2r-\eps}-1)}\le\sum_{k=0}^{n-1}\big(s_{k+1}^2-s_k^2\big)\le \frac {c(2^{(2-2r+\eps)n}-1)}{(2^{2-2r+\eps}-1)}.$$ 
Cancelling the terms in the telescopic sum, taking logarithms, dividing by $2n$, passing to the limit $n\ua\infty$, and sending $\eps$ to zero now yields the result.
\end{proof}

The following corollary follows immediately from \Cref{Thm_main} and \Cref{weighted QV prop}.

\begin{corollary}\label{Corollary QV}
	If $x \in C[0,1]$ satisfies the reverse Jensen condition, stated in \Cref{Definition reverse Jensen condition}, and one of the conditions in \ref{weighted QV prop part a}, \ref{weighted QV prop part b} or \ref{weighted QV prop part c} in \Cref{weighted QV prop} holds for some $r \in (0,1)$, then $x$ admits the roughness exponent $r$.
\end{corollary}

The implications \ref{weighted QV prop part c}$\Rightarrow$\ref{weighted QV prop part b} and 
\ref{weighted QV prop part b}$\Rightarrow$ \ref{weighted QV prop part a} in \Cref{weighted QV prop} are generally false. This is illustrated in the following example. 

\begin{example}\label{couterexample weight quadratic}	For given $R\in(0,1)$, we consider $x\in C[0,1]$ defined through the Faber--Schauder coefficients  $\theta_{m,k}=2^{m(1/2-R)}$ if $m$ is even and $\theta_{m,k}=0$ if $m$ is odd. Recall from \Cref{Takagi class rem} that $x$ satisfies the reverse Jensen condition. A straightforward computation gives 
	$$\wh R_n(x)=1-\frac1{2n}\log_2\frac{2^{4+4(1-R)\lfloor\frac{n-1}2\rfloor}-2^{4R}}{2^4-2^{4R}}\lra R.
	$$
	\Cref{Thm_main}  hence yields that $x$ has the roughness exponent $R$. However, $\wh R^*_n(x)$ clearly does not converge. It follows that also none of the quantities in part \ref{weighted QV prop part a}
	of \Cref{weighted QV prop} can converge.
	
	Now take again $R\in(0,1)$ and let $\theta_{m,k}=2^{m(1/2-R)}\sqrt{(2^{2(1-R)}-1)g(m)}$, where the function $g$ is strictly positive and regularly varying with index $\varrho$. It was shown in Corollaries 3.6 and 3.11 of~\cite{HanSchiedZhang2} that the corresponding function $x$ has the roughness exponent $R$ and that the  $(1/R)^{\rm th}$ variation of $x$ is zero for $\varrho<0$ and infinite for $\varrho>0$. If $g$ is slowly varying, so that $\varrho=0$, then $\liminf_n \<x\>_n^{(1/R)}$  and $\limsup_n \<x\>_n^{(1/R)}$ can take arbitrary values in $[0,\infty]$. We clearly have
	$\wh R^*_n(x)=R-\frac1{2n}\log_2\big((2^{2-2R}-1)g(n)\big)\to R$.
	However, $2^{n(2R-2)}(s_{n+1}^2-s_n^2)=(2^{2-2R}-1)g(n)$
	does not converge to a finite limit unless $g(n)$ does. This shows that the implication \ref{weighted QV prop part b}$\Rightarrow$\ref{weighted QV prop part a} in \Cref{weighted QV prop} does not hold.
\end{example}

As announced above, we now present an example of a continuous function $x$ for which the weighted quadratic variation  $2^{n(2\gamma-1)}\<x\>_n^{(2)}$ converges to a finite limit for some  $\gamma \in (0,1)$, but which nevertheless admits a different roughness exponent $R$. Note that  \Cref{weighted QV prop} implies that $\wh R_n(x)\to \gamma$. In other words, the limit of $\wh R_n(x)$ is different from the function's roughness exponent. Hence, $x$ cannot satisfy the reverse Jensen condition. The function $x$ therefore also highlights the importance of the reverse Jensen condition in \Cref{Thm_main}.

\begin{proposition}\label{remark Jensen counter}
	Let $\alpha \in (0,1]$ and $R \in (0,1)$. Then the function 
	\begin{equation*}
		x = \sum_{m = 0}^\infty 2^{\left(\frac{1}{2}-\alpha R\right)m}\sum_{k = 0}^{\floor{2^{\alpha m}}-1} e_{m,k}
	\end{equation*}
	admits the roughness exponent $R$. On the other hand, let $\gamma:= \alpha R + (1-\alpha)/2$, then
	\begin{equation}\label{eq weight qv}
		\lim_{n \ua \infty} 2^{n(2\gamma - 1)}\<x\>^{(2)}_n = \left(2^{2 - 2\gamma} - 1\right)^{-1} \quad \text{and} \quad \lim_{n \ua \infty}\wh R_n(x) = \gamma.
	\end{equation}
\end{proposition}

It is worthwhile to point out that $\gamma = \alpha R + (1-\alpha)/2$ is the convex combination of $R$ and $1/2$. If $R > 1/2$, we have $R \ge \gamma > 1/2$ and vice versa. In fact, such a relation holds universally as later shown in \Cref{Prop_Bias}.

\begin{proof}[Proof of \Cref{remark Jensen counter}]
	We begin with proving \eqref{eq weight qv}. To this end, we verify condition (iii) in \Cref{weighted QV prop} \ref{weighted QV prop part a}:
	\begin{equation*}
		\lim_{n \ua \infty}2^{n(2\gamma-2)}\sum_{k = 0}^{2^n-1}\theta^2_{n,k} = \lim_{n \ua \infty}2^{n\left((2\gamma-2)+(1-2\alpha R)\right)}\floor{2^{\alpha n}} = 1.
	\end{equation*}
	\Cref{weighted QV prop} then yields that $\lim_n \wh R_n(x) = \gamma$. 
	
	Next, we aim to show that $x$ admits the roughness exponent $R$. For a fixed $\alpha \in (0,1]$, let $n_0 \in \bN$ be such that $\floor{2^{\alpha n}}/2 \le \floor{2^{\alpha (n-1)}}-1/2$ for $n \ge n_0$. Let 
	\begin{equation*}
		\wt x := \sum_{m = n_0}^\infty 2^{\left(\frac{1}{2}-\alpha R\right)m}\sum_{k = 0}^{\floor{2^{\alpha m}}-1} e_{m,k}.
	\end{equation*}
		Based on the argument presented in the proof of Lemma 3.1 in \cite{MishuraSchied2},  the functions $x$ and $\widetilde{x}$ share the same roughness exponent. Hence, it will be sufficient to henceforth consider the function $\widetilde{x}$. Recalling the definition of $V_m$ and $S_m$ in \eqref{eq_Sm_Def}, it is clear that for $m \ge n_0$ and $2^mV_m \ge \floor{2^{\alpha m}}$, we have $S_m = 2^{m(2-2\alpha R)}$, and otherwise zero.  Clearly, 
	$
	2^{m-1}V_{m-1}  \le 2^{m-1}V_m$. Therefore, for $m \ge n_0$, if $2^{m}V_m \le \floor{2^{\alpha m}}-1$, then 
	\begin{equation}\label{eq prob lower}
		2^{m-1}V_{m-1} \le \frac{2^{m}V_m}{2} \le \frac{1}{2}\left(\floor{2^{\alpha m}}-1\right) \le \floor{2^{\alpha (m-1)}}-1.
	\end{equation}
	In other words, if $S_m = 2^{m(2-2\alpha R)}$, then $S_{m-1} = 2^{(m-1)(2-2 \alpha R)}$. On the other hand, for $m \ge n_0$, if $2^{m}V_m \ge \floor{2^{\alpha m}}$, then 
	\begin{equation}\label{eq prob upper}
		2^{m+1}V_{m+1} \ge 2^{m+1}V_m \ge 2\floor{2^{\alpha m}} \ge 2\floor{2^{\alpha m}} - 1 \ge \floor{2^{\alpha (m+1)}}.
	\end{equation}
	Equivalently speaking, if $S_m = 0$, then $S_{m+1} = 0$. For $n \ge n_0 + 3$, it then follows from \eqref{eq prob lower} and \eqref{eq prob upper} that if there exists $n_0 + 1 \le m \le n-2$ such that $2^mV_m \le \floor{2^{\alpha m} }-1$ and $2^{m+1}V_{m+1} \ge \floor{2^{\alpha (m+1)}}$, we have $S_k = 2^{k(2-2\alpha R)}$ for $n_0 \le k \le m$ and $S_k =0$ for  $m + 1 \le k \le n-2$. Therefore, for $n \ge n_0 + 3$, the distribution of the random variable $\sum_{m = 0}^{n-1}S_m$ is as follows, 
	\begin{align*}
		\bP\bigg(\sum_{m = 0}^{n-1}S_m = \sum_{m = n_0}^{n-1}2^{m(2-2\alpha R)}\bigg)= \frac{\floor{2^{\alpha (n-1)}}}{2^{n-1}},\qquad \bP\bigg(\sum_{m = 0}^{n-1}S_m = 0\bigg) = 1 - \frac{\floor{2^{\alpha n_0}}}{2^{n_0}},
	\end{align*}
	and 
	\begin{align*}
		\bP\bigg(\sum_{m = 0}^{n-1}S_m = \sum_{m = n_0}^{k-1}2^{m(2-2\alpha R)}\bigg)  &= \frac{\floor{2^{\alpha(k-1)}}}{2^{k-1}} - \frac{\floor{2^{\alpha k}}}{2^k}\qquad \text{for $n_0+1 \le k \le n-1$.}
	\end{align*} Therefore, for $n \ge n_0 + 3$ and $p \ge 1$, 
	\begin{equation*}
		\begin{split}
			\bE\bigg[\bigg(\sum_{m = 0}^{n-1}S_m\bigg)^{p/2}\bigg] &= \frac{\floor{2^{\alpha (n-1)}}}{2^{n-1}}\bigg(\sum_{m = n_0}^{n-1}2^{m(2-2\alpha R)}\bigg)^{p/2}+\sum_{k = n_0+1}^{n-1}\bigg(\frac{\floor{2^{\alpha(k-1)}}}{2^{k-1}} - \frac{\floor{2^{\alpha k}}}{2^k}\bigg)\bigg(\sum_{m = n_0}^{k-1}2^{m(2-2\alpha R)}\bigg)^{p/2}.
		\end{split}
	\end{equation*}
	In the next step, we seek suitable upper and lower bounds for  this expression. To obtain  a lower bound, we use the following facts, 
	\begin{equation*}
		\frac{\floor{2^{\alpha (n-1)}}}{2^{n-1}} \ge 2^{(\alpha - 1)(n-1)} -2^{-(n-1)}  \quad \text{and} \quad \sum_{m = n_0}^{n-1}2^{m(2-2\alpha R)} \sim \frac{2^{n(2-2\alpha R)}}{2^{2-2\alpha R}-1} \quad \text{as} \quad n \ua \infty.
	\end{equation*}
	Hence, there exists a positive constant $c_1$ such that for sufficiently large $n$, 
	\begin{equation}\label{remark Jensen counter lower bound eq}
		\bE\bigg[\bigg(\sum_{m = 0}^{n-1}S_m\bigg)^{p/2}\bigg] \ge \left( 2^{(\alpha - 1)(n-1)} -2^{-(n-1)}\right)\bigg(\sum_{m = n_0}^{n-1}2^{m(2-2\alpha R)}\bigg)^{p/2} \ge c_1 2^{n(p(1-\alpha R)-(1-\alpha))}.
	\end{equation}
	To obtain the upper bound, we apply the inequality $\floor{2^{\alpha(n-1)}}/2^{n-1} \le 2^{(\alpha - 1)(n-1)}$ and Jensen's inequality, 
	\begin{equation*}
			\bE\bigg[\bigg(\sum_{m = 0}^{n-1}S_m\bigg)^{p/2}\bigg] \le \big(n^{p/2-1}\vee 1\big)\sum_{k = n_0+1}^{n}2^{(\alpha - 1)(k-1)}\bigg(\sum_{m = n_0}^{k-1}2^{mp(1-\alpha R)}\bigg).
	\end{equation*}
	As $p(1 - \alpha R) > 1 - \alpha \ge 0$, the following asymptotics hold as $n \ua \infty$, 
	\begin{equation*}
		\sum_{k = n_0+1}^{n}2^{(\alpha-1)(k-1)}\sum_{m = n_0}^{k-1}2^{mp(1-\alpha R)} = O \big(2^{n(p(1-R)-(1-\alpha))}\big).
	\end{equation*}
	Thus, for each $p \ge 1$, there exists a positive constant $c_2$ such that for sufficiently large $n$, 
	\begin{equation}\label{remark Jensen counter upper bound eq}
		\bE\bigg[\bigg(\sum_{m = 0}^{n-1}S_m\bigg)^{p/2}\bigg] \le c_2\left(n^{p/2-1}\vee 1 \right)2^{n(p(1-\alpha R)-(1-\alpha))}.
	\end{equation}
	It then follows from \eqref{remark Jensen counter lower bound eq} and \eqref{remark Jensen counter upper bound eq} that there exists a positive constant $c_p$ depending on $p$ such that for sufficiently large $n$, 
	\begin{equation*}
		\frac{1}{c_p} 2^{n(p(1-\alpha R)-(1-\alpha))} \le \bE\bigg[\bigg(\sum_{m = 0}^{n-1}S_m\bigg)^{p/2}\bigg] \le c_p\left(n^{p/2-1}\vee 1 \right)2^{n(p(1-\alpha R)-(1-\alpha))}.
	\end{equation*}
	\Cref{Proposition_Burk} now yields that for each $p \ge 1$, there exists $n_0 \in \bN$, such that for $n \ge n_0$, 
	\begin{equation*}
		\frac{A_p}{c_p} 2^{\alpha n(1-pR)} \le \<\wt x\>^{(p)}_n \le c_pB_p 2^{\alpha n(1-pR)}.
	\end{equation*}
	 Taking $n \ua \infty$ on both sides of the above inequalities gives
	\begin{equation*}
		\lim_{n \ua \infty}\<\wt x\>^{(p)}_n = \begin{cases}
			0 \quad &\text{for} \quad p > \frac{1}{R},\\
			\infty \quad &\text{for} \quad p < \frac{1}{R}.\\
		\end{cases}
	\end{equation*}
	Thus, $\wt x$ admits the roughness exponent $R$, which completes the proof.
\end{proof}
		
\section{Application to fractional Brownian motion with drift}\label{Gaussian examples section}

Let $(B^H_t)_{t \in [0,1]}$ be a fractional Brownian motion with Hurst parameter $H\in(0,1)$, defined on the probability space $(\wt \Omega,\cG,\bQ)$. Let $X^H$ be given by 
\begin{equation}\label{eq sde fbm}
	X^H_t := x_0 + B^H_t + \int_{0}^{t}\zeta_s \,ds \qquad \text{for} \quad 0 \le t \le 1,
\end{equation}
where $\zeta$ is progressively measurable with respect to the natural filtration of $B^H$ and satisfies the following additional assumptions.
If $H\le1/2$, we assume  that $t\mapsto\zeta_t$ is $\bQ$-a.s.~bounded in the sense that  there exists a finite random variable $C$ such that $\zeta_t(\om)\le C(\om)$ for a.e.~$t$ and $\bQ$-a.s.~$\om\in\wt \Omega$. 
	 If $H>1/2$, we assume that $t\mapsto \zeta_t$ is $\bQ$-a.s.~H\"older continuous with some exponent $\alpha>2H-1$.
This class of stochastic processes includes solutions of stochastic integral equations of the form
$X^H_t=x_0+B^H_t+\int_0^t b(X^H_s)\,ds$ 
 for some locally bounded and measurable function  $b:\bR\to\bR$ that, for $H>1/2$, is locally H\"older continuous with some exponent $\alpha>2-1/H$. One particular example is the  fractional Ornstein--Uhlenbeck process, which was proposed by Gatheral et al.~\cite{GatheralRosenbaum} as a suitable model for log-volatility.

\begin{theorem}\label{Theorem example FBM}
	For all $H\in(0,1)$, almost all sample paths  of the process $X^H$   satisfy  the reverse Jensen condition and admit the roughness exponent $H$. Moreover, the estimator $\wh R_n$ is strongly consistent and there exists a positive constant $\kappa$ such that 
	\begin{equation}
	\limsup_{n\ua\infty}\frac{|\wh R_n(X^H) - H| }{\eps_n}\le \kappa\qquad\text{$\bQ$-a.s.,}
	\end{equation}
	where 
	\begin{equation}\label{eq rate ori path}
		\eps_n= \begin{cases}
			2^{-n/2}n^{-1}\sqrt{\log n} &\text{if $H \in (0,\frac{1}{2})$,} \\
			2^{-n/2}n^{-1/2}\sqrt{\log n} &\text{if $H = \frac{1}{2}$,}\\
			2^{(H-1)n}n^{-1}\sqrt{\log n} &\text{if $ H \in (\frac{1}{2},1)$.} 
		\end{cases} 
	\end{equation}
\end{theorem}

For fractional Brownian motion with Hurst parameter $H\le1/2$, it follows from Theorem 1.2 in~\cite{marcus1992p} that $ \lim_n
\<B^H\>_n^{(1/H)}=\pi^{-1/2}{2^{1/(2H)}\Gamma(\frac{H+1}{2H})}$  $\bQ$-a.s. As discussed at the beginning of \Cref{general section}, this implies  that $B^H$ admits $\bQ$-a.s.~the roughness exponent $H$. So this part of the assertion of \Cref{Theorem example FBM} was known beforehand. For $H>1/2$,  however,  we were only able to find references in the literature that yield convergence of $\<B^H\>_n^{(1/H)}$ in probability  (see, e.g.,  Section 1.18 in~\cite{Mishura} and the references therein), which is not quite sufficient to conclude that almost all trajectories of $B^H$ admit the roughness exponent~$H$.

The rate of convergence for $\wh R_n(B^H)$  was initially studied in~\cite[Theorem~1]{KubiliusMelichov2010} for $H >1/2$. However, as pointed out in our communication~\cite{KestutisCommunication} with the authors, the convergence rate stated in~\cite{KubiliusMelichov2010} contains a typo. By modifying and streamlining the original proof idea from~\cite{KubiliusMelichov2010}, our \Cref{Theorem example FBM} provides the correct rate, states a stronger convergence result,  extends it to the cases $H=1/2$ and $H<1/2$, and also covers the case of fractional Brownian motion with absolutely continuous drift. We start with the following proposition, which provides the speed of convergence in Gladyshev's theorem~\cite{Gladyshev} and is of possible independent interest. 

\begin{proposition}\label{Gladyshev speed prop} For $\eps_n$ as in \eqref{eq rate ori path} and any $H\in (0,1)$, there exists a constant $c$ such that 
$$\limsup_{n\ua\infty}\frac{\big|2^{n(2H-1)}\<B^H\>_n^{(2)}-1\big|}{n\eps_n}\le c\qquad\text{$\bQ$-a.s.}
$$
\end{proposition}

\begin{proof}Let $Z_n $ be the $2^n$-dimensional centered Gaussian vector  with components
	$		Z_n^k = 2^{n(H-1/2)} (B^H_{{k}{2^{-n}}}-B^H_{(k-1){2^{-n}}} )$. Then $\norm{Z_n}^2_{\ell_2} =2^{n(2H-1)}\<B^H\>^{(2)}_n$. 
		Denote by $\Gamma_n$ the covariance matrix of  $Z_n $ and by $T_n$ its trace. According to the following concentration inequality from~\cite[Lemma 3.1]{BaudoinHairer},  there exists a constant $\lambda$ such that for all $n$ and each $\delta>0$,
	\begin{equation*}
		\bQ\left[\big|\norm{Z_n}_{\ell_2} - T_n\big| \ge \delta \right] \le \lambda \exp\Big(-\frac{\delta^2}{4\norm{\Gamma_n}_2}\Big),
	\end{equation*}
	where $\norm{\Gamma_n}_2$ is the spectral norm of $\Gamma_n$.  It was shown in~\cite[Equation (15)]{KleinGine} that $\trace \Gamma_n = 1$, and $\norm{\Gamma_n}_2 \le \gamma b_n$, where $\gamma$ is a positive constant and $b_n=n^2\eps_n^2/\log n$.
	Taking $\delta_n:=2\sqrt{2\gamma b_n\log n}=n\eps_n\sqrt{8\gamma}$ yields $\bQ [|\norm{Z_n}_{\ell_2} - 1| \ge \delta_n]\le\lambda/n^2$, which is summable.  The Borel--Cantelli lemma hence implies that $|\norm{Z_n}_{\ell_2} - 1| \le\delta_n$ eventually with probability one, and so 
	$$\limsup_{n\ua\infty}\frac{\big|\sqrt{2^{n(2H-1)}\<B^H\>_n^{(2)}}-1\big|}{\delta_n}\le1\qquad\text{$\bQ$-a.s.}
$$
Using the fact that $\sqrt{1+x}-1\sim x/2$ for $x\to0$ now yields the result.
\end{proof}

\begin{proof}[Proof of \Cref{Theorem example FBM}]
	 It is shown in~\cite{HanSchiedFBM} that that law of $(X^H_t)_{t \in [0,1]}$ is absolutely continuous with respect to the law of $(x_0 + B^H_t)_{t \in [0,1]}$. Hence, it suffices to prove the  assertion for fractional Brownian motion.

	We start by establishing the reverse Jensen condition. Consider the probability space $(\Om,\cF,\bP)$ supporting the random variables $( V_m)$ from \eqref{eq_Sm_Def}.  Then $(S_m)_{m\in\bN} = (2^m\theta^2_{m,2^mV_m})_{m\in\bN}$
	is a sequence of random variables defined on  the product space $(\Omega \times \wt \Omega, \cF\otimes\cG, \bP \otimes \bQ)$.
	According to \eqref{Thm_main cond}, we must show that for every $p\ge1$, with $\bQ$-probability one,   there exists a subexponential function $\varrho_p$ and $n_p\in\bN$ such that for $n\ge n_p$,
	\begin{equation}\label{Theorem example FBM to show}
		\frac1{\varrho_p(n)} \bE\Big[2^{2n(H-1)}\sum_{m=0}^{n-1}S_m\Big]^{p/2}\le \bE\bigg[\bigg(2^{2n(H-1)}\sum_{m=0}^{n-1}S_m\bigg)^{p/2}\bigg]\le\varrho_p(n) \bE\Big[2^{2n(H-1)}\sum_{m=0}^{n-1}S_m\Big]^{p/2},
	\end{equation}
	where the notation $ \bE$ denotes as usual the expectation with respect to $\bP$. 
	To this end, we note that $\bQ$-a.s.~as $n \ua \infty$,
	\begin{equation}\label{Gladyshev conv eq}
		\bE\Big[2^{2n(H-1)}\sum_{m=0}^{n-1}S_m\Big]^{p/2}= \bigg(2^{n(2H-1)}2^{-n}\sum_{m = 0}^{n-1}\sum_{k = 0}^{2^m-1}\theta^2_{m,k}\bigg)^{p/2}=\big(2^{n(2H-1)}\<B^H\>^{(2)}_n\big)^{p/2}\lra 1,
	\end{equation}
	where the second identity follows from \eqref{QV eq} and the convergence from  \Cref{Gladyshev speed prop}.
	The self-similarity of fractional Brownian motion implies that 
	\begin{equation}\label{Rogers identity}
		Y_n:=2^{n(pH-1)}\<B^H\>_n^{(p)}=2^{n(pH-1)}\sum_{k=0}^{2^n-1}\big|B^H_{(k+1)2^{-n}}-B^H_{k2^{-n}}\big|^p\sim 2^{-n}\sum_{k=0}^{2^{n}-1}\big|B^H_{k+1}-B^H_k\big|^p,\end{equation}
	where $\sim$ means equality in law. 	
	Since the increments $B^H_{k+1}-B^H_k$ are standard normally distributed, the $\bQ$-expectation of  $Y_n$ is bounded from above by a finite constant $c$, uniformly in $n$.
	It follows that $\bQ[Y_n>n^2]\le c/n^2$, and a Borel--Cantelli argument yields that $Y_n\le n^2$  for sufficiently large $n$, $\bQ$-a.s. In conjunction with \eqref{eq_burkholder}	and \eqref{Gladyshev conv eq},  this establishes the right-hand inequality in \eqref{Theorem example FBM to show} with $\varrho_p(n)=2n^2/A_p$.
	For the left-hand inequality, we take $q>p$ and note that by \eqref{Rogers identity}
	and Jensen's inequality,
	\begin{align*}
		Y_n^{-q}\sim \bigg(2^{-n}\sum_{k=0}^{2^{n}-1}|B^H_{k+1}-B^H_k|^p\bigg)^{-q}\le 2^{-n}\sum_{k=0}^{2^{n}-1}|B^H_{k+1}-B^H_k|^{-p/q} .\end{align*}
	Since the increments $B^H_{k+1}-B^H_k$ are standard normally distributed, the $\bQ$-expectations of $Y_n^{-q}$ are also bounded from above, uniformly in $n$, as $-p/q > -1$. A similar Borel--Cantelli argument as above then yields that $Y_n\ge 1/n^{2q}$  for sufficiently large $n$, $\bQ$-a.s., and we obtain the left-hand inequality in \eqref{Theorem example FBM to show} with $\varrho_p(n)=2B_pn^{2q}$. Thus, both inequalities in \eqref{Theorem example FBM to show}  hold, e.g.,  with $\varrho_p(n):=2(A_p^{-1}\vee B_p)n^{2(p+1)}$. This establishes the reverse Jensen condition. 
	
	Next, it follows from Propositions \ref{Gladyshev speed prop} and \ref{weighted QV prop} that $\wh R_n(B^H)\to H$ $\bQ$-a.s., and so \Cref{Thm_main}  yields that the sample paths admit the roughness exponent $H$ with probability one.

	Finally,  we prove the rate of convergence \eqref{eq rate ori path}. Using \eqref{QV eq}, one checks that 
	$$H-\wh R_n(B^H)=\frac1{2n}\log_2 2^{n(2H-1)}\<B^H\>_n^{(2)}.
	$$
Now the assertion follows from 	\Cref{Gladyshev speed prop} and the fact that $\log_2(1+x)\sim x/\log2$ for $x\to0$.
\end{proof}

\section{The connection with Besov regularity}\label{Besov section}

In this section, we establish a connection between the roughness exponent and Besov regularity. We refer to~\cite[Section 2.3]{TriebelFunctionSpace} for the definition of the Besov space $\cB^{r}_{p,q}[0,1]$. Ciesielski et al.~\cite{Ciesielski1993}   provided a norm  expressed in terms of  Faber--Schauder coefficients, which is  equivalent to the traditional Besov norm. Rosenbaum~\cite{Rosenbaum2009} defined the following equivalent norm  on $\cB^{r}_{p,q}[0,1]$ for $p,q \in [1,\infty)$ and $r \in (1/p,1)$,
\begin{equation}\label{Rosenbaum norm}
\|x\|_{r,p,q}:=|x(0)|\vee\|(a_n)_{n\in\bN_0}\|_{\ell_q},\qquad \text{where}\qquad  a_n:=\big( 2^{n(pr-1)}\<x\>_n^{(p)}\big)^{1/p}.
\end{equation} 
Note that for $p=2$ the relation to weighted quadratic variation as studied in \Cref{weighted QV section} becomes apparent. Measuring the degree of roughness of a function $x \in [0,1]$ through Besov spaces was initially proposed by Eyink~\cite{EyinkBesov} through the condition
\begin{equation}\label{Eyink condition}
	x \in \bigcap_{\varepsilon > 0} \cB^{r - \varepsilon}_{p,\infty}[0,1]\setminus\bigcup_{\varepsilon > 0}\cB^{r + \varepsilon}_{p,\infty}[0,1].\
	\end{equation}
 
\begin{example}\label{Besov example 1}Consider the function $x$ defined in \Cref{remark Jensen counter} with $\alpha \in (0,1)$ and $R > 1/2$. The first identity in \eqref{eq weight qv} implies that $\|x\|_{\gamma,2,\infty}<\infty$ and that $\|x\|_{\gamma+\eps,2,\infty}=\infty$  for any $\eps>0$. Hence, condition \eqref{Eyink condition} holds for $r=\gamma$, $p=2$, and $q=\infty$. Nevertheless, as observed in \Cref{remark Jensen counter}, the roughness exponent of $x$ is different from $\gamma$, which shows that Eyink's concept of Besov regularity is generally different from our concept of roughness. 
\end{example}

In the rough volatility context~\cite{GatheralRosenbaum}, the parameter $r$ in \eqref{Eyink condition} is often assumed to be the same for all $p \ge 1$, which leads to the requirement
\begin{equation}\label{eq Besov 1}
	x \in \bigcap_{{\varepsilon > 0,\ p \ge 1}} \cB^{r - \varepsilon}_{p,\infty}[0,1]\setminus\bigcup_{{\varepsilon > 0,\  p \ge 1}}\cB^{r + \varepsilon}_{p,\infty}[0,1].
\end{equation}
In the following theorem, we explore the connection between this concept of Besov regularity and the roughness exponent. Let $q^\ast$ be as in \eqref{q+q- eqn}  and  $R^-:= 1/q^\ast$. Our result shows that the condition \eqref{eq Besov 1} implies that $r=R^-$. As a matter of fact, the condition \eqref{eq Besov cond} used in \Cref{Theorem Besov}
 is actually less restrictive than \eqref{eq Besov 1}. Note moreover that \Cref{remark Jensen counter} shows that \eqref{eq Besov 1} or \eqref{eq Besov cond} cannot be replaced with Eyink's condition \eqref{Eyink condition}. Indeed, the latter condition holds in \Cref{remark Jensen counter} for $r=\gamma$, while the function $x$ constructed there admits an arbitrary roughness exponent $R = R^-$.

\begin{theorem}\label{Theorem Besov}
	Suppose that there exists $r \in (0,1)$ such that 
	\begin{equation}\label{eq Besov cond}
		x \in \bigcap_{{\varepsilon > 0,\ p \ge 1}} \cB^{r - \varepsilon}_{p,\infty}[0,1]\setminus\bigcup_{{\varepsilon > 0,\  \frac{1}{r+\varepsilon} < p < \frac{1}{r}}}\cB^{r + \varepsilon}_{p,\infty}[0,1].
	\end{equation}
	Then $R^- = r$.
\end{theorem}
\begin{proof}
	We must show that $q^* = 1/r$, and we start with proving $q^* \le 1/r$. For $\gamma > 0$ and $0 < \varepsilon < r$, we choose $p:= (1+\gamma)/(r-\varepsilon)$. Then it is clear that $p(r-\varepsilon) = 1 + \gamma > 1$.Hence, 	$(\sup_{n}2^{\gamma n}\<x\>^{(p)}_n\big)^{1/p} =\|x\|_{r-\eps,p,\infty} < \infty$. 
	Since $\gamma > 0$, we get 
	$
		\limsup_{n}\<x\>^{(p)}_n < \infty
	$,
	which implies that $q^* \le p=(1+\gamma)/(r-\varepsilon)$. Taking $\gamma,\varepsilon \da 0$ yields $q^* \le 1/r$. 
	
	Next, let us show that $q^* \ge 1/r$. To this end, we assume by way of contradiction that $q^* < 1/r$. Let us denote $q = \frac{1}{2}(q^* + 1/r)$. Then there exists $\lambda > 0$ such that $q =1/(r + \lambda)$ as $q < 1/r$. Next, we set 
	\begin{equation*}
		\varepsilon := \frac{r}{8} - \frac{r^2q}{8} = \frac{r}{8}\left(1 - \frac{r}{r+\lambda}\right) = \frac{r\lambda}{8(r + \lambda)} < \lambda.
	\end{equation*} As $q^* < q$, we  have
	$		\limsup_{n \ua \infty}\<x\>^{(q)}_n = 0$.	Let $H^\alpha[0,1]$ denote the space of $\alpha$-H\"{o}lder continuous functions on $[0,1]$. It follows from a fundamental embedding result~\cite[Section 2.7.1]{TriebelFunctionSpace} that $\cB^r_{p,\infty}[0,1] \subset H^{r-1/p}[0,1]$, and this yields
	\begin{equation*}
		\bigcap_{{\varepsilon > 0,\ p \ge 1}}\cB^{r - \varepsilon}_{p,\infty}[0,1]\subset \bigcap_{\varepsilon > 0} H^{r-\varepsilon}[0,1].
	\end{equation*}
	Next, we take $p := 1/(r + \varepsilon/2)$. Then $p > q$ as $\varepsilon < \lambda$. Applying the above H\"older condition gives
	$	
			\<x\>^{(p)}_n\le \kappa 2^{(q-p)(r-\varepsilon)n}\<x\>^{(q)}_n$
	for some $\kappa > 0$. Thus, for $n \in \bN$, we get
	$	2^{n((r+\varepsilon)p-1)}\<x\>^{(p)}_n \le \kappa 2^{n((r+\varepsilon)p-1 + (q-p)(r-\varepsilon))}\<x\>^{(q)}_n
	$. One checks that $
			(r+\varepsilon)p-1 + (q-p)(r-\varepsilon) < 0$. 	Hence,
	\begin{equation*}
		\limsup_{n \ua \infty}2^{n((r+\varepsilon)p-1)}\<x\>^{(p)}_n \le \kappa \left(\limsup_{n \ua \infty}2^{n((r+\varepsilon)p-1 + (q-p)(r-\varepsilon))}\right)\left(\limsup_{n \ua \infty}\<x\>^{(q)}_n\right) = 0.
	\end{equation*}
	Therefore $\norm{x}_{r+\varepsilon,p,\infty} < \infty$ for $p = 1/(r+\varepsilon/2)$, which contradicts the underlying assumption that $x \notin \cB^{r + \varepsilon}_{p,\infty}[0,1]$, and thus $q^* \ge 1/r$. 
\end{proof}

In our next example, we construct a function $x$ for which condition \eqref{eq Besov cond} of Besov regularity is satisfied but where $\wh R_n(x)$ oscillates between a given number $r\in(0,1)$ and $1$. In this sense, the function exhibits an entire spectrum of roughness, which is a phenomenon akin to multifractality. After a preprint version of the present paper was posted on arXiv, Das \cite{Dasthesis} and Cont and Das~\cite{ContDas24} proposed to use $R^-$ as a \lq\lq roughness index" for measuring the roughness of a trajectory $x\in C[0,1]$. 
Based on the previous discussion, it is clear that this roughness index captures only one side of the roughness spectrum $[R^-,R^+]$ for $R^+:=1/q_*$, whereas the existence of the roughness exponent requires  $R^-=R^+$. Moreover, \Cref{Theorem Besov} shows that also the condition of Besov regularity only captures $R^-$.  In the following example, we analyze a concrete illustration of a situation in which $R^-<R^+$. Nevertheless, as shown in \Cref{xipm remark}, the Gladyshev estimator $\wh R_n$ can characterize both ends of the roughness spectrum under the reverse Jensen condition.

\begin{example}\label{example Besov} For $r\in(0,1)$, consider the function $x$ with Faber--Schauder coefficients $\theta_{n,k}=2^{(1/2-r)n}$ if $n=m!$ for some $m\in\bN$ and  $\theta_{n,k}=0$ otherwise. The idea of constructing such a function is based on Faber~\cite{Faber07}.  	It is easy to see that 
	\begin{equation*}
		2^{(2-2r)n!} \le s^2_{n!+1} \le n 2^{(2-2r)n!} \quad \text{and} \quad 2^{(2-2r)(n-1)!} \le s^2_{n!} \le (n-1) 2^{(2-2r)(n-1)!},
	\end{equation*}
	and hence $r^+ = 1$ and $r^- = r$, where $r^\pm$ is as in \eqref{eq xi pm}. For each $m \in \bN$, the Faber--Schauder coefficient $\theta_{m,k}$ does not depend on $k$, and so  $x$ satisfies the reverse Jensen condition by \Cref{Takagi class rem}. Thus, we have $q^\ast = 1/r^- = 1/r$ and $q_\ast = 1/r^+ = 1$, and so
	$$R^-=\frac1{q^*}=r\quad\text{and}\quad R^+=\frac1{q_*}=1
	$$
	by \Cref{xipm remark}. Consequently, $x$ does not admit a roughness exponent.
	Let us now show that
	\begin{equation}\label{example Besov claim eq}
		x \in \bigcap_{{\varepsilon > 0,\  p \ge 1}} \cB^{r - \varepsilon}_{p,\infty}[0,1]\setminus\bigcup_{{\varepsilon > 0,\  \frac{1}{r+\varepsilon} < p < \frac{1}{r}}}\cB^{r + \varepsilon}_{p,\infty}[0,1].
	\end{equation}
	We prove first  that $x$ belongs to the intersection on the left-hand side.
	For every $p \ge 1$, it follows from \eqref{eq p-2 var} that there exist $0 < A_p \le B_p < \infty$ and a subexponential function $\varrho_p$ such that 
	\begin{equation*}
		A_p\varrho_p^{-1}(n)2^{n(1-p)}s^p_n \le \<x\>_n^{(p)} \le B_p\varrho_p(n)2^{n(1-p)}s^p_n.
	\end{equation*} 
	Since $r \in (0,1)$, we have that for any $\varepsilon > 0$ and $p \ge 1$, 
	\begin{equation*}
		\begin{split}
			\norm{x}_{r-\varepsilon,p,\infty} &=  	\sup_{n}\Big(2^{n(p(r-\varepsilon)-1)}\<x\>^{(p)}_n\Big)^{1/p} \le    \sup_{n} \Big(2^{np(r-\varepsilon-1)}B_p \varrho_p(n)s_n^p\Big)^{1/p} \\ &=  \sup_{n } \Big(2^{(n!+1)p(r - \varepsilon - 1)}B_p\varrho_p(n!+1)s^p_{n!+1}\Big)^{1/p} \\&\le  \sup_{n} \Big(2^{(n!+1)p(r - \varepsilon - 1)}B_p\varrho_p(n!+1)n 2^{p(1-r)n!}\Big)^{1/p}\\&=  2^{r - \varepsilon - 1}B^{1/p}_p \sup_n \Big(2^{-p\varepsilon n!}n\varrho_p(n!+1)\Big)^{1/p} < \infty.
		\end{split}
	\end{equation*}The final inequality holds because $\varrho_p(n)$ is a sub-exponential function. In the next step, we are going to show that for any $\varepsilon > 0$ and $1/r > p > 1/(r+\varepsilon)$, we have $x \notin \cB^{r+\varepsilon}_{p,\infty}[0,1]$. Since $p > 1/(r+\varepsilon)$, 	\begin{equation*}
		\begin{split}
			\norm{x}_{r+\varepsilon,p,\infty} & =  	\sup_{n}\bigg(2^{n(p(r+\varepsilon)-1)}\<x\>^{(p)}_n\bigg)^{1/p} \ge    \sup_{n } \bigg(2^{(n!+1)p(r + \varepsilon - 1)}A_p\varrho^{-1}_p(n!+1)s^p_{n!+1}\bigg)^{1/p}\\
			& \ge  \sup_{n} \bigg(2^{(n!+1)p(r + \varepsilon - 1)}A_p\varrho^{-1}_p(n!+1) 2^{p(1-r)n!}\bigg)^{1/p}\\&= 2^{r + \varepsilon - 1}A_p^{1/p}\sup_n \bigg(2^{p\varepsilon n!}\varrho^{-1}_p(n!+1)\bigg)^{1/p},
		\end{split}
	\end{equation*}
	which is infinite, because $\varrho_p$ is sub-exponential. 
\end{example}

\section{The roughness exponent without the reverse Jensen condition}\label{gen section}

In this section, we collect some general mathematical properties of the
roughness exponent, which hold even if the reverse Jensen condition is not satisfied. As in \Cref{general section}, we fix an
arbitrary function $x \in C[0,1]$ with Faber--Schauder coefficients
$(\theta_{m,k})$ and $\wh R_n(x)$ as
defined in \eqref{eq_zeta_def}. Moreover, $r^+ = \limsup_n \wh R_n(x)$ and $r^- = \liminf_n \wh R_n(x)$ are as in \eqref{eq xi pm}. We start with the
following a priori estimates linking the roughness exponent $R$ with $r^+$ and $r^-$.

\begin{proposition}\label{Prop_Bias} Suppose that $x$ admits the roughness exponent $R$. Then the
	following assertions hold:
	\begin{enumerate}
		\item \label{Bias_a} If  $r^+ < 1/2$, then $ R\le1/2$; if  $r^->1/2$, then
		$R\ge1/2$.
		\item \label{Bias_b} If $R \le 1/2$, then $R \le r^-$; if $R \ge 1/2$, then  $R \ge r^+$.
	\end{enumerate}
\end{proposition}

\begin{proof}
	Recall the short-hand notion $r_n:= \wh R_n(x)$. To prove \ref{Bias_a}, we suppose first that $r^+ < 1/2$. Since \eqref{QV eq} states that $\<x\>_n^{(2)}=2^{n(1 - 2r_n)}$ and $\liminf_n \<x\>^{(2)}_n = \infty$, we must have $R\le 1/2$. In the same way, we get $R\ge1/2$ if
	$r^- > 1/2$. To prove  \ref{Bias_b}, we suppose first that $R \le 1/2$ and take $p > 1/R \ge 2$. Applying
	Jensen's inequality to \eqref{eq_Inq} gives
	\begin{equation}\label{eq_Jensen_1}
		\begin{split}
			\<x\>_n^{(p)}  &\ge A_p2^{n(1-p)}\bE\Big[\Big(\sum_{m = 0}^{n-1}S_m\Big)^{p/2}\Big] \ge A_p2^{n(1-p)}\bE\Big[\Big(\sum_{m = 0}^{n-1}S_m\Big)\Big]^{p/2} = A_p2^{n(1-pr_n)}.
		\end{split}    
	\end{equation}
	Since $\limsup_n \<x\>_n^{(p)}  = 0$, we must have $r^- > 1/p$. Taking $1/p
	\ua R$ gives $R \le r^-$. The assertion for $R \ge 1/2$ can be proved analogously.
\end{proof}

The preceding proposition provides one-sided bounds on the roughness
exponent in terms of~$r:= \lim_n \wh R_n(x)$. Our next result gives  universal two-sided
bounds for $R$. To this end, for $m \in \bN$, we denote by
\begin{equation*}
	\widehat{F}_m(t):= 2^{-m}\Big(\sum_{k = 0}^{2^m-1}\Ind{\{\theta_{m,k}^2\le t2^{-m}\}}\Big)
\end{equation*}
the empirical distribution of the $m^{\text{\rm th}}$ generation
Faber--Schauder coefficients and let $\widehat{F}^{-1}_m$ be a corresponding
quantile function.

\begin{proposition}\label{prop_two_bound} Suppose that $x$ admits the roughness exponent $R$. We define for $\nu \in \bN$ and $n \in \bN$,
	\begin{equation*}
		\begin{split}
			 r_{\nu,n}^{+}&:= 1- \frac{1}{2n}\log_2\sum_{m =
				0}^{n-1}\frac{1}{m^\nu}\sum_{k =1
			}^{m^\nu}\widehat{F}^{-1}_m\Big(2^{-m}\Big\lfloor\frac{2^m(k-1)}{m^\nu}\Big\rfloor\Big),
			\\  r_{\nu,n}^{-}&:= 1-\frac{1}{2n}\log_2\sum_{m =
				0}^{n-1}\frac{1}{m^\nu}\sum_{k =1
			}^{m^\nu}\widehat{F}^{-1}_m\Big(2^{-m}\Big\lceil\frac{2^{m}k}{m^\nu}\Big\rceil\Big).
		\end{split}
	\end{equation*}
	If the limits $r^{\pm}_{\nu}:= \lim\limits_{n \ua
		\infty}r^{\pm}_{\nu,n}$ exist for some  $\nu\in\bN$, then $r_\nu^- \le R \le r_\nu^+$.
\end{proposition}

\begin{proof}[Proof of \Cref{prop_two_bound}]
	Let us first  consider the case $p \ge 2$. It follows from
	\eqref{eq_Jensen_1} that
	\begin{equation*}
		\<x\>_n^{(p)}  \ge A_p2^{n(1-p)}\bE\Big[\sum_{m = 0}^{n-1}S_m\Big]^{p/2} =  A_p2^{n(1-p)}\bigg(\sum_{m = 0}^{n-1}\int_{0}^{1}\widehat{F}^{-1}_m(t)dt\bigg)^{p/2}.
	\end{equation*}
	Furthermore, we have 
	\begin{equation*}
		\int_{0}^{1}\widehat{F}^{-1}_m(t)dt = \sum_{k = 1}^{m^\nu}\int_{(k-1)/m^\nu}^{k/m^\nu}\widehat{F}^{-1}_m(t)dt.
	\end{equation*}
	Note that $2^{-m}\floor{2^m(k-1)/m^\nu} \le (k-1)/m^\nu \le k/m^\nu \le
	2^{-m}\ceil{2^mk/m^\nu}$. Moreover, both $2^{-m}\floor{2^m(k-1)/m^\nu}$
	and $2^{-m}\ceil{2^mk/m^\nu}$ belong to $\{k2^{-m}:k=0,\dots,2^m\}$.
	Hence,
	\begin{equation}\label{eq_two_sided_interval_1}
		\widehat{F}^{-1}_m\Big(2^{-m}\Big\lceil{\frac{2^{m}k}{m^\nu}}\Big\rceil\Big) \ge \widehat{F}^{-1}_m(t) \ge \widehat{F}^{-1}_m\Big(2^{-m}\Big\lfloor{\frac{2^m(k-1)}{m^\nu}}\Big\rfloor\Big)
	\end{equation}
	for all $t \in [(k-1)/m^\nu,k/m^\nu]$. This gives
	\begin{equation*}
		\<x\>_n^{(p)}  \ge A_p2^{n(1-p)}\Big[\sum_{m = 0}^{n-1}\frac{1}{m^\nu}\sum_{k =1 }^{m^\nu}\widehat{F}^{-1}_m\Big(2^{-m}\Big\lfloor{\frac{2^m(k-1)}{m^\nu}}\Big\rfloor\Big)\Big]^{p/2} = A_p2^{n(1-pr^{+}_{\nu,n})}.
	\end{equation*}
	Moreover, applying Jensen's inequality to \eqref{eq_burkholder}, we get 
	\begin{equation*}
		\<x\>_n^{(p)}  \le B_p2^{n(1-p)}n^{p/2}\bE\Big[\Big(\frac{1}{n}\sum_{m = 0}^{n-1}S_m\Big)^{p/2}\Big] \le B_p2^{n(1-p)}n^{p/2-1}\sum_{m = 0}^{n-1}\bE[S_m^{p/2}].
	\end{equation*}
	We once again apply \eqref{eq_two_sided_interval_1} to this inequality
	and obtain
	\begin{equation*}
		\begin{split}
			\<x\>_n^{(p)}  
			&\le B_p2^{n(1-p)}n^{p/2-1}\left(\sum_{m = 0}^{n-1}\frac{1}{m^\nu}\sum_{k = 1}^{m^\nu}\Big[\widehat{F}^{-1}_m\Big(2^{-m}\Big\lceil{\frac{2^{m}k}{m^\nu}}\Big\rceil\Big)\Big]^{p/2}\right)\\
			&\le B_p2^{n(1-p)}n^{(\nu+1)(p/2-1)}\Big[\sum_{m = 0}^{n-1}\frac{1}{m^\nu}\sum_{k = 1}^{m^\nu}\widehat{F}^{-1}_m\Big(2^{-m}\Big\lceil{\frac{2^{m}k}{m^\nu}}\Big\rceil\Big)\Big]^{p/2}\\&= B_p2^{n(1-pr^-_{\nu,n}+(\nu+1)(p/2-1)\log_2 n/n )}.
		\end{split}
	\end{equation*}
	Moreover, a similar inequality can be obtained for $1 \le p \le 2$,
	\begin{equation*}
		A_p2^{n\big(1-pr^{+}_{\nu,n}+(\nu+1)(p/2-1)\log_2 n/n\big)}\le \<x\>_n^{(p)}  \le B_p2^{n(1-pr_{\nu,n}^{-})}.
	\end{equation*}
	In both cases ($p \ge 2$ or $1 \le p \le 2$), the exponents inside the
	brackets converge to $1 - p r^\pm_\nu$ as $n \ua \infty$. Therefore,
	$\lim_{n}\<x\>_n^{(p)}  = \infty$ for $p < 1/r^+_\nu$ and $\lim_{n}\<x\>_n^{(p)}  = 0$ for $p > 1/r^-_\nu$. This leads to
	$r^-_\nu \le R \le r_\nu^+$, and concludes the proof.
\end{proof}

Our next  result  looks at the special case of functions $x$ of bounded
variation. Such functions clearly have the roughness exponent $R=1$. The
converse statement, however, is not true: there exist functions with $R=1$ that
are nowhere differentiable and hence not of bounded variation; see \Cref{TL ex}. 

\begin{proposition}\label{BV prop} The following assertions hold: 
	\begin{enumerate}
		\item \label{Cor_BV_a} If $\sup_n s_n < \infty$, then the function
		$x$ is of bounded variation. 
		\item \label{Cor_BV_b} If the function $x$ is of bounded variation,
		then $\sup_n 2^{-n/2}\sum_{k = 0}^{2^n-1}|\theta_{n,k}| < \infty$.
	\end{enumerate}
\end{proposition}

\begin{proof}
	\ref{Cor_BV_a}: It follows from \Cref{Proposition_Burk} that there
	exists $B_1 > 0$ such that 
	\begin{equation*}
		\<x\>_n^{(1)}  \le B_1\bE\Big[\Big(\sum_{m = 0}^{n-1}S_m\Big)^{1/2}\Big] \le B_1\Big(\sum_{m = 0}^{n-1}\bE[S_m]\Big)^{1/2} = B_1 s_n.
	\end{equation*}
	Therefore, $\sup_n\<x\>_n^{(1)} \le B_1\sup_n s_n < \infty$, and
	$\sup_n\<x\>_n^{(1)} $  coincides with the total variation of the continuous
	function $x$ (see, e.g., Theorem 2 in \S5 of Chapter VIII in
	\cite{Natanson}).
	
	\ref{Cor_BV_b}: Following from \Cref{Proposition_Burk}, we have 
	\begin{equation*}
		A_1\Big(2^{-n/2}\sum_{k = 0}^{2^n-1}|\theta_{n,k}|\Big) = A_1\bE[S_n^{1/2}] \le A_1\bE\Big[\Big(\sum_{m = 0}^{n-1}S_m\Big)^{1/2}\Big] \le\<x\>_n^{(1)} ,
	\end{equation*}
	for some $A_1 > 0$. Taking suprema on both sides completes our proof.
\end{proof}

\section{Model-free estimation of the roughness exponent}\label{stat section}
	
	In this section, we will introduce and analyze model-free estimators for the roughness exponent. To this end, recall that	$\widehat{R}_n(x) = 1-\frac1n\log_2 s_n$ is a consistent estimator for the roughness exponent of any function  $x \in C[0,1]$ whose Faber--Schauder coefficients $(\theta_{m,k})$ satisfy the reverse Jensen condition. \Cref{fig:three graphs} illustrates that the estimator performs very well for sample paths of fractional Brownian motion and yields  accurate estimates of the corresponding Hurst parameter. 
			\begin{figure}[H]
		\centering
			\includegraphics[width=8cm]{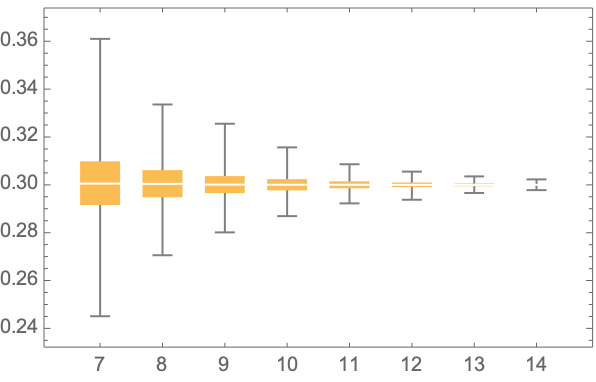}
		\quad		\includegraphics[width=8cm]{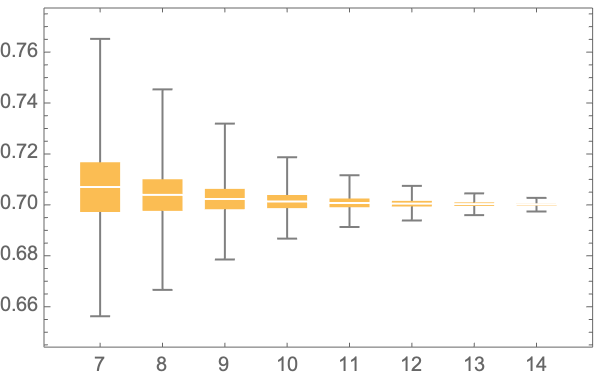}
		\caption{Box plots of $\wh R_n$ for $n=7,\dots, 14$ based on  10,000 sample paths of fractional Brownian motion  with  $R = 0.3$ (left) and $R = 0.7$ (right). }
		\label{fig:three graphs}
	\end{figure}

But the estimator $\wh R_n$ also has a disadvantage: it is not scale-invariant. Indeed,  we clearly have		\begin{equation}\label{wh H scaling eq}
			\widehat{R}_n(\lambda x) - \widehat{R}_n(x) = -\frac{\log_2|\lambda|}{n}\qquad\text{for $\lambda\neq0$.}
		\end{equation}
As a consequence, multiplying $x$ with a constant factor can lead to negative or unreasonably large estimates for $R$ and 
substantially slow down or speed up the convergence $\wh R_n(x)\to R$ (note that both the roughness exponent itself and  the reverse Jensen condition are scale-invariant).  
		Naive normalization of the  data does not provide a resolution of this issue, as is illustrated in  \Cref{normalized whHn figure}. In the next subsection, we are therefore going to derive intrinsically defined optimal scaling factors and the corresponding improved estimators for the roughness exponent.

\begin{figure}[H]
		\centering
		\begin{minipage}[b]{0.45\textwidth}
		\includegraphics[width=\textwidth]{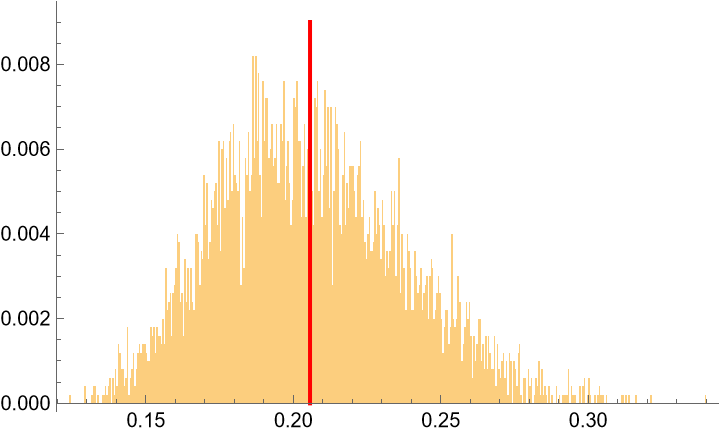}
		\end{minipage}
		\hfill
		\begin{minipage}[b]{0.45\textwidth}
			\centering
			\includegraphics[width=\textwidth]{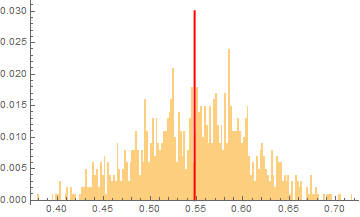}
		\end{minipage}
\caption{Histograms of $\wh R_{12}$ applied to 5000 standardized sample paths of fractional Brownian motion with Hurst parameters  $H=0.3$ (left) and $H=0.7$ (right). Here, standardization means that each sample path is affinely scaled to have empirical mean zero and variance one. The red vertical lines represent the empirical means  of the corresponding estimates.  }\label{normalized whHn figure}
		\end{figure}

When presented with a generic time series, we face significant uncertainty about possible data scaling that, as discussed above, can lead to significant underperformance of the estimator $\wh R_n$. 
In this section, we will therefore analyze two ideas on constructing scale-invariant estimators for the roughness exponent out of the sequence $(\wh R_n)$. Both ideas will lead to new estimators that are  linear combinations of $\wh R_n,\dots, \wh R_{n-m}$ for a number $m$.  

	The first idea consists in  looking for a scaling factor $\lambda_n$ that optimizes a certain criterion over all possible parameters $\lambda>0$ and then to take $\wh R_n(\lambda_n x)$ as a new estimator. In particular, we will look at the following two choices. A third choice will be given in \Cref{regression scale est def}.
	 
	 \begin{definition}\label{sls and tls def}Fix $m\in\bN$ and $\alpha_0,\dots,\alpha_{m}\ge0$ with $\alpha_0>0$.
	 \begin{enumerate}
	 \item For $n>m$, the \textit{sequential scaling factor} $\lambda_{n}^{s}$ and the \textit{sequential scale estimate} $R^{s}_{n}(x)$ are defined as follows,
	\begin{equation}\label{eq_def_seqloc}
		\begin{split}
			\lambda^{s}_{n}&:= \argmin_{\lambda > 0} \sum_{k = n-m}^{n}\alpha_{n-k}\Big(\widehat{R}_k(\lambda x) - \widehat{R}_{k-1}(\lambda x)\Big)^2 \quad \text{and} \quad
			R^{s}_{n}(x):= \wh R_{n}(\lambda^{s}_{n} x).
		\end{split}
	\end{equation}
	The corresponding mapping $R^{s}_{n}:C[0,1]\rightarrow \bR$ will be called the \textit{sequential scale estimator}.
\item The \textit{terminal scaling factor} $\lambda_{n}^{t}$ and the \textit{terminal scale estimate} $R_{n}^{t}(x)$ are defined as follows, 
	\begin{equation}\label{eq_def_terloc}
		\begin{split}
			\lambda^{t}_n&:= \argmin_{\lambda > 0} \sum_{k = n-m}^{n}\alpha_{n-k}\Big(\widehat{R}_n(\lambda x) - \widehat{R}_{k-1}(\lambda x)\Big)^2 \quad \text{and} \quad
			R_{n}^{t}(x):= \wh R_n(\lambda^{t}_n x).
		\end{split}
	\end{equation}
The corresponding mapping $R_n^{t}:C[0,1]\rightarrow \bR$ will be called the \textit{terminal scale estimator}.
 \end{enumerate}
	 \end{definition}

We will see in 	 \Cref{Thm_LE} that  \Cref{sls and tls def} is well posed, because the minimization problems in   \eqref{eq_def_seqloc}
 and \eqref{eq_def_terloc}  admit unique solutions for every function $x\in C[0,1]$. 
	The intuition for  sequential scaling is fairly simple. The sequential scaling factor $\lambda^{s}_{n}$ minimizes the weighted mean-squared differences $\widehat{R}_k(\lambda x) - \widehat{R}_{k-1}(\lambda x)$ for $k = m+1,\dots,n$. Since $\widehat{R}_n$ is a consistent estimator, the faster the sequence $\widehat{R}_k(\lambda x)$ converges, the more accurate the final estimate $\widehat{R}_n(\lambda x)$ will be. The idea for  terminal scaling is also straightforward: Since the estimator $\widehat{R}_n$ is consistent,  the last estimate $\widehat{R}_n(x)$ should be more precise than any other previous value  $\widehat{R}_k(x)$ with $k < n$. Therefore, we aim at minimizing the weighted mean-squared differences $\widehat{R}_k(\lambda x)$ and $\widehat{R}_n(\lambda x)$.

	Our second  idea for constructing a scale-invariant estimator is as follows. It follows from \eqref{wh H scaling eq} that for $\lambda>0$
	$$n\wh R_n( x)=n\wh R_n(\lambda x)+\log_2\lambda.
	$$
For an ideal scaling factor  $\lambda$, the right-hand side will be close to $n R+\log_2|\lambda|$. Therefore, the idea is to perform linear regression on the data points $(n-m)\wh R_{n-m}(x),(n-m-1)\wh R_{n-m-1}(x),\dots, n\wh R_n( x)$ and to take the corresponding slope as an estimator for $R$.

\begin{definition}\label{regression scale est def}Let $m\in\bN$ and  $\alpha_0,\dots,\alpha_m\ge0$ be such that $\sum_{k=0}^m\alpha_k=1$. Then, for $n>m$, the  \emph{regression scale estimate} $R^{r}_n(x)$ and the \emph{regression scaling factor} $\lambda^{r}_n$ are given by 
\begin{equation}\label{regression problem}
(R^{r}_n(x),\lambda^{r}_n)=\argmin_{h\in\bR,\,\lambda>0}\sum_{k=0}^m\alpha_k\Big((n-k)(\wh R_{n-k}(x)-h)-\log_2\lambda\Big)^2.
\end{equation}
The corresponding mapping	$R_n^{r}:C[0,1]\rightarrow \bR$ will be called the \textit{regression scale estimator}.
\end{definition}

	The next proposition shows in particular that  all three estimators can be represented by linear combinations of the estimators $\widehat{R}_k$ for $m \le k \le n$ and that they are scale-invariant  estimators of the roughness exponent $R$.

	\begin{proposition}\label{Thm_LE} Consider the context of Definitions~\ref{sls and tls def} and~\ref{regression scale est def} with fixed $m\in\bN$ and $\alpha_0,\dots,\alpha_{m}\ge0$ such that $\alpha_0>0$.
	\begin{enumerate}
	\item The three optimization problems in \eqref{eq_def_seqloc}, \eqref{eq_def_terloc}, and \eqref{regression problem} admit unique solutions for every function $x\in C[0,1]$. In particular, all objects in Definitions~\ref{sls and tls def} and~\ref{regression scale est def}  are well defined.
 \item The sequential and terminal scale estimators can be represented  by the following respective linear combinations of the estimators $\wh R_k$,
  \begin{align*}
   R^{s}_{n} =\beta_{n,n}\wh R_n+\cdots+\beta_{n,n-m-1}\wh R_{n-m-1}\qquad\text{and}\qquad  R^{t}_{n} =\gamma_{n,n}\wh R_n+\cdots+\gamma_{n,n-m-1}\wh R_{n-m-1}
  \end{align*}
where
$$\beta_{n,k}=\begin{cases}\displaystyle1+\frac{\alpha_0}{c^{\textrm{\rm s}}_{n}n^2(n-1)}&\text{if $k=n$,}\\
\displaystyle\frac1{c^{\textrm{\rm
s}}_{n}nk}\Big(\frac{\alpha_{n-k}}{k-1}-\frac{\alpha_{n-k-1}}{k+1}\Big)&\text{if
$n-m\le k\le n-1$,}\\
\displaystyle
\frac{-\alpha_m}{c^{\textrm{\rm s}}_{n}n(n-m)(n-m-1)}&\text{if $ k= n-m-1$,}
\end{cases}\quad\text{for}\quad c^{s}_{n}:= \sum_{k =
n-m}^{n}\frac{\alpha_{n-k}}{k^{2}(k-1)^{2}},
$$
and
$$\gamma_{n,k}=\begin{cases}\displaystyle1+\frac1{c_n^{t}n^2}\sum_{j=n-m}^n\alpha_{n-j}\frac{n-j+1}{j-1}&\text{if $k=n$,}\\
\displaystyle\frac{k-n}{c_n^{t}n^2k}\alpha_{n-k-1}&\text{otherwise,}
\end{cases}
\quad\text{for}\quad
c_n^{t}=\sum_{k=n-m}^n\alpha_{n-k}\Big(\frac{n-k+1}{n(k-1)}\Big)^2.
$$
If in addition $\sum_k\alpha_k=1$, then the regression scale estimator is given by
\begin{equation}\begin{split}
R^{r}_n=\frac1{c^{r}}\sum_{k=0}^m\alpha_k(n-k)\big(k-a)\wh R_{n-k},\qquad\text{where $\displaystyle a=\sum_{k=0}^m\alpha_kk$ and $\displaystyle c^{r}=a^2-\sum_{k=0}^m\alpha_kk^2$.}\end{split}
\end{equation}

\item The sequential, terminal, and regression scale estimators are scale-invariant. That is,  for $n>m$, $x \in C[0,1]$, and $\lambda \neq 0$, we have $R^{t}_{n}(\lambda x) = R^{t}_{n}(x)$, $R^{s}_{n}(\lambda x) = R^{s}_{n}(x)$ and $R^{r}_{n}(\lambda x) = R^{r}_{n}(x)$.
	\item If $x\in C[0,1]$ is such that there exists $\lambda\neq 0$ for which $|\wh R_n(\lambda x)-R|=O(a_n)$ as $n\ua\infty$ for some sequence $(a_n)$ with $a_n=o(1/n)$, then $			|R^{s}_{n}(x)-R| $, $ |R^{t}_{n}(x) -R|$, and $|R^{r}_{n}(x) -R|$ are all of the order $O(na_n)$.
	\end{enumerate}			\end{proposition}

	\begin{proof}
	We prove (a) and (b) together. Fix $x \in C[0,1]$. For $\lambda>0$ we write $\phi:=\log_2\lambda$ and we will minimize \eqref{eq_def_seqloc}
 and \eqref{eq_def_terloc} over $\phi$ rather than over $\lambda$. Recall from  \eqref{wh H scaling eq}
  that 
		\begin{equation}\label{eq_parametrization}
			\widehat{R}_k(2^\phi x)= \widehat{R}_k(x)-\frac\phi k.
		\end{equation}
		One therefore sees that the objective functions in \eqref{eq_def_seqloc}
 and \eqref{eq_def_terloc} are strictly convex in $\phi$, and so minimizers  can be computed as  the unique zeros of the corresponding derivatives. Differentiating the objective function in \eqref{eq_def_seqloc} and summing by parts gives		\begin{align}		
			\lefteqn{	\frac12\frac{\partial}{\partial \phi}\sum_{k = n-m}^{n}\alpha_{n-k}\Big(\widehat{R}_k(2^\phi x) - \widehat{R}_{k-1}(2^\phi x)\Big)^2}\nonumber\\
		&=
				c^{s}_{n}\phi+ \frac{\alpha_0\wh R_n(x)}{n(n-1)}+\sum_{k=n-m}^{n-1}\wh R_k(x)\Big(\frac{\alpha_{n-k}}{k(k-1)}-\frac{\alpha_{n-k-1}}{(k+1)k}\Big)-\frac{\alpha_m\wh R_{n-m-1}(x)}{(n-m)(n-m-1)}.\label{summation by parts eq}
			\end{align}
		Setting this expression equal to zero yields the global minimizer $\phi^{s}_n$,
		\begin{equation*}
			\phi^{s}_n = -\frac{1}{c^{s}_{n}}\bigg(\frac{\wh R_n(x)}{n}\cdot \frac{\alpha_0}{n-1}+\sum_{k=n-m}^{n-1}\frac{\wh R_k(x)}{k}\Big(\frac{\alpha_{n-k}}{k-1}-\frac{\alpha_{n-k-1}}{k+1}\Big)-\frac{\wh R_{n-m-1}(x)}{n-m-1}\cdot\frac{\alpha_m}{n-m}\bigg).		\end{equation*}
	Applying \eqref{eq_parametrization}
to $\wh R_n(2^	{\phi_n^{s}} x)$ yields the asserted expression for  the sequential scale estimator. 		
	The proof for the terminal scale estimator is analogous. The assertion for the regression scale estimator follows by standard computations for linear regression. 
	
(c) For $\eta > 0$, we have 
		\begin{equation*}
			\argmin_{\lambda > 0} \sum_{k = n-m}^{n}\alpha_{n-k}\Big(\widehat{R}_k(\lambda\eta x) - \widehat{R}_{k-1}(\lambda \eta x)\Big)^2 = \frac{1}{\eta}\argmin_{\lambda > 0} \sum_{k = n-m}^{n}\alpha_{n-k}\Big(\widehat{R}_k(\lambda x) - \widehat{R}_{k-1}(\lambda x)\Big)^2 = \frac{\lambda^{s}_{n}}{\eta},
		\end{equation*}
		where $\lambda^{s}_{n}$ is the sequential scaling factor. Therefore, 
	$
			R^{s}_{n}(\eta x) =  \widehat{R}_n(\lambda^{s}_{n}x) = R^{s}_{n}(x)$.		The same argument yields the scale invariance of the  terminal and regression scale estimators. 
		
		(d) Clearly, we can assume without loss of generality that $\lambda=1$. First, we consider the sequential scale estimator. Using \eqref{summation by parts eq}
one sees that  $\sum_k\beta_{n,k}=1$ so that 
$$  R^{s}_{n} -R=\beta_{n,n}(\wh R_n-R)+\cdots+\beta_{n,n-m-1}(\wh R_{n-m-1}-R).
$$
Since $\beta_{n,n-k}=O(n)$ for each $k$, the assertion follows. The proof for the terminal scale estimator is analogous. For the regression scale estimator, one checks that $\frac1{c^{r}}\sum_{k=0}^m\alpha_k(n-k)\big(k-a)=1$ and proceeds as before.
\end{proof}

Part (d) of \Cref{Thm_LE} implies in particular the consistency of the sequential, terminal, and regression scale estimators if $x\in C[0,1]$ satisfies $|\wh R_n(\lambda x)-R|=o(1/n)$ for some $\lambda\neq0$. It also enables us to obtain the convergence rate of the estimators $R^s_n$, $R^t_n$ and $R^r_n$ for the sample paths of fractional Brownian motion with drift.

\begin{corollary}\label{thm fbm convergence} Let $X^H$ be a fractional Brownian motion with drift as in \eqref{eq sde fbm}.
	Then the following almost sure rates of convergence hold for the scale-invariant estimator $R_n^p$ with  $p\in\{ s,t,r\}$,
		\begin{equation}\label{eq rate scale}
		\left|R^p_n(X^H) - H\right| = \begin{cases}
			  O \left(2^{-n/2}\sqrt{\log n}\right) &\text{if $H \in (0,\frac{1}{2})$,} \\
				  O \left(2^{-n/2}n^{1/2}\sqrt{\log n}\right) &\text{if $ H = \frac{1}{2}$,}\\
			  O \left(2^{(H-1)n}\sqrt{\log n}\right) &\text{if $ H \in (\frac{1}{2},1)$.} 
		\end{cases} 
	\end{equation}
\end{corollary}

		\begin{remark}
		A general scale-invariant estimator $\wt R_n$ can be constructed by solving an optimization problem of the  form
		$$\widetilde{\lambda}:= \argmin_{\lambda > 0}\sum_{j > k}^{n}\alpha_{j,k}|\widehat{R}_{j}(\lambda x) - \widehat{R}_k(\lambda x)|^2 ,$$
		where $\alpha_{j,k}$ are given coefficients, and by setting $
		\widetilde{R}_n(x) := \widehat{R}_n(\widetilde{\lambda}x)$. If the
		global minimizer $\widetilde{\lambda}$ exists,  the estimate
		$\widetilde{R}_n(x)$ will be a linear combination of the estimates
		$\widehat{R}_k(x)$.
	\end{remark}

Let us now investigate the relations between our estimators and the one used in
\cite{GatheralRosenbaum}. To this end, we let for $x \in C[0,1]$, $n,k \in \bN$
and $q \in \bR^-$, 
	\begin{equation*}
		m(q,k,n):= \Big\lfloor{\frac{k}{2^n}}\Big\rfloor^{-1}\sum_{j = 1}^{\floor{2^{n}/k}}|x(kj2^{-n})-x(k(j-1)2^{-n})|^q.
	\end{equation*}
	Following Rosenbaum~\cite{Rosenbaum2009}, Gatheral et
	al.~\cite{GatheralRosenbaum} assume that there exists a positive constant
	$b_q$  such that
	\begin{equation*}
		\big(k2^{-n}\big)^{-qR} m(q,k,n) \rightarrow b_q, \qquad 
		\text{for any fixed $k \in \bN$, as $n \ua \infty$,} 
	\end{equation*}
From here, the estimator $R_n^v$ from~\cite{GatheralRosenbaum} is computed by
way of a  linear regression. More precisely, let $\mathcal{K}$ be a finite
collection of positive integers and $\mathcal{Q}$ be a finite collection of
positive real numbers. For $q \in \mathcal{Q}$, the estimate $R_{n,q}^{v}(x)$ of
the roughness exponent of the continuous function $x$ using the
$q^\text{th}$ variation is obtained by regressing $\log_2 m(q,k,n)$ with respect
to $\log_2 k$, i.e.,
	\begin{equation}\label{eq simple regression}
		(R_{n,q}^{v}(x),b_q) = \argmin_{h\in\bR,\,b_q>0} \sum_{k \in \mathcal{K}}\Big(\log_2 m(q,k,n)-\log_2 b_q - qh(\log_2 k - n)\Big)^2.
	\end{equation}
	Standard computations for linear regression then yield that
	\begin{equation*}
		R_{n,q}^{v}(x)= \dfrac{\sum_{k \in \mathcal{K}}\big(\log_2 k - \overline{\log_2 \mathcal{K}}\big)\big(\log_2 m(q,k,n)-\overline{\log_2 m(q,\mathcal{K},n)}\big)}{q \cdot \sum_{k \in \mathcal{K}}\big(\log_2 k - \overline{\log_2 \mathcal{K}}\big)},
	\end{equation*}
	where $|\cdot|$ denotes the cardinality of a set  and 
	$$\overline{\log_2 \mathcal{K}} = \frac1{ |\mathcal{K}|}\sum_{k \in
	 \mathcal{K}}\log_2 k\quad\text{and}\quad \overline{\log_2
	 m(q,\mathcal{K},n)} =\frac1{ |\mathcal{K}|}\sum_{k \in \mathcal{K}}\log_2
	 m(q,k,n).$$ The \textit{simple regression estimate} $R_n^v(x)$ is then the
	 sample average of $R_{n,q}^{v}(x)$ for all $q \in \mathcal{Q}$. In other
	 words, we have
	\begin{equation*}
		R_{n}^{v}(x):= \frac{1}{|\mathcal{Q}|}\sum_{q \in \mathcal{Q}}R_{n,q}^{v}(x).
	\end{equation*}
	The corresponding mapping $R_n^v: C[0,1] \rightarrow \bR$ will be called the
	\textit{simple regression estimator}, which is a formal description of the
	estimator used in~\cite{GatheralRosenbaum}. It is also clear that the simple
	regression estimator $R^v_n$ is scale-invariant, i.e., for any $n \in \bN$
	and $x \in C[0,1]$, we have $R^v_n(\lambda x) = R^v_n(x)$ for any $\lambda
	\neq 0$. Our next result states that the simple regression estimator is a
	particular case of our regression scale estimator under a certain choice of
	parameters.
	
	\begin{proposition}\label{Gatheral prop}Let $\mathcal{Q} = \{2\}$,
	$\mathcal{K} = \{1,2,\cdots,2^m\}$ and $\alpha_k = \alpha$ for all $k$ and
	some $\alpha > 0$. Then the regression scale estimator $R^r_n$ coincides
	with the simple regression estimator $R^v_n$ for all $n>m$.
	\end{proposition}
	
	\begin{proof}
		Since the minimization problems in \eqref{regression problem} and
		\eqref{eq simple regression} admit unique solutions for every function
		$x \in C[0,1]$,  it suffices to show \eqref{regression problem} is
		equivalent to \eqref{eq simple regression} when taking the above
		parameters. Let $\lambda =1/\sqrt{b_2}$, and we rewrite optimization
		problem in \eqref{eq simple regression} into
		\begin{equation*}
		\begin{split}
			&\argmin_{h\in\bR,\,b_q>0}\sum_{k \in \mathcal{K}}\Big(\log_2 m(2,k,n)-\log_2 b_2 - 2h(\log_2 k - n)\Big)^2  \\&= \argmin_{h\in\bR,\,\lambda>0}\sum_{j = 0}^{m}\Big(2(j-n)+2\log_2s_{n-j}+2\log_2 \lambda - 2h(j - n)\Big)^2\\&=\argmin_{h\in\bR,\,b_q>0} \sum_{j =0}^{m}\alpha \cdot\Big((n-j)(\widehat{R}_{n-j}(x)-h)-\log_2 \lambda \Big)^2,
		\end{split}
		\end{equation*}
	where the second identity follows from~\cite[Proposition
	2.1]{MishuraSchied}.\end{proof}

\section{The roughness exponent for unequally spaced
partitions}\label{irregular section}

In this section, we will relax our assumption that the data points for sampling
the $p^{\text{th}}$ variation $\<x\>_n^{(p)}$ of $x$ are located at the
$n^{\text{th}}$ dyadic partition of the time interval $[0,1]$. To this end, let
us first recall that a \emph{partition} of the interval $[0,1]$ is a finite set
$\bT=\{t_0,\dots, t_m\}$ such that $0=t_0<t_1<\cdots<t_m=1$. Its \emph{mesh} is
defined as the maximal distance between two neighboring partition points. A
sequence $(\bT_n)_{n\in\bN_0}$ of partitions of $[0,1]$ is called a
\emph{refining sequence of partitions of $[0,1]$}  if
$\bT_0\subset\bT_1\subset\cdots$ and the mesh of $\bT_n$ tends to $0$ as
$n\ua\infty$. In this case, we will generically write
$\bT_n=\{t_0^n,t^n_1,\dots, t_{m_n}^n\}$. For simplicity, we will henceforth
assume that $\bT_0=\{0,1\}$ and that, when passing from $\bT_n$ to $\bT_{n+1}$,
exactly one new partition point is added between any two neighboring partition
points of $\bT_n$, i.e., $|(t_{k-1}^n,t^n_k)\cap \bT_{n+1}|=1$.  This assumption
can always be satisfied by re-arranging the partitions, and in this case, we
have  $m_n=2^n$. Furthermore, under this assumption, it follows $t^n_{k} =
t^{n+1}_{2k}$.  

In \Cref{general section}, the definition of the roughness exponent of a
 continuous function $x$ was based on the Faber--Schauder expansion of $x$. The
 Faber--Schauder functions are clearly tied to the sequence of dyadic
 partitions, but, similar to Cont and Das~\cite{ContDas21}, the definition of
 the Faber--Schauder functions can be extended to our present setup as follows.
 First, let $b_{\emptyset}(t):=t$. Next, for $n\in\bN_0$ and $k\in\{0,\dots,
 2^{n}-1\}$, we define
$$b_{n,k}(t):=\begin{cases}\displaystyle\frac{t-t^{n+1}_{2k}}{t^{n+1}_{2k+1}-t^{n+1}_{2k}}&\text{if
	$t\in[t^{n+1}_{2k},t^{n+1}_{2k+1})$,}\\ &\\\displaystyle
	\frac{t^{n+1}_{2k+2}-t}{t^{n+1}_{2k+2}-t^{n+1}_{2k+1}}&\text{if
	$t\in[t^{n+1}_{2k+1},t^{n+1}_{2k+2})$,}\\ & \\
	0&\text{otherwise.}
\end{cases}
$$
It is shown in~\cite[Section 3.3]{ContDas21} that any  function $x \in C[0,1]$ can by represented by the following uniformly convergent series, 
\begin{equation}\label{finite basis development}
		{x}(t)={f}(0)+(x(1)-x(0))t+\sum_{m=0}^{\infty}\sum_{k=0}^{2^{m}-1}\theta_{m,k}b_{m,k}(t),\qquad t\in[0,1],
\end{equation}
where \begin{equation}\label{Faber--Schauder coefficients}
	\theta_{n,k}:=x(t^{n+1}_{2k+1})-\frac{t^{n+1}_{2k+2}-t^{n+1}_{2k+1}}{t^{n+1}_{2k+2}-t^{n+1}_{2k}}x(t^{n+1}_{2k})-\frac{t^{n+1}_{2k+1}-t^{n+1}_{2k}}{t^{n+1}_{2k+2}-t^{n+1}_{2k}}x(t^{n+1}_{2k+2})
\end{equation} 
will be called the \emph{generalized Faber--Schauder coefficients of $x$.} For $p > 0$ and $n \in \bN$, we define the $p^\text{th}$ variation of the function $x$ along  $\bT_n$ by
\begin{equation*}
	\<x\>_n^{(p)}:= \sum_{k = 0}^{2^n-1}\big|x(t^n_{k+1})-x(t^n_{k})\big|^p. 
\end{equation*} If there exists $q \in [1,\infty]$ such that 
\begin{equation*}
	\lim\limits_{n \ua \infty} \<x\>_n^{(p)} = \begin{cases}
		0 &\quad p > q,\\
		\infty &\quad p < q,
	\end{cases}
\end{equation*}
the \emph{roughness exponent} of $x$ (with respect to $(\bT_n)_{n \in \bN_0}$) is  defined as $R = 1/q$. The following definition formulates the \text{reverse Jensen condition} for the irregular Faber--Schauder coefficients defined in \eqref{Faber--Schauder coefficients}. 

\begin{definition}\label{def converse Jensen condition}
	We say that the coefficients $(\theta_{m,k})$ of $x$ satisfy the \textit{reverse Jensen condition} if for each $p \ge 1$, there exists a non-decreasing subexponential function $\varrho_p: \bN \rightarrow [1,\infty)$ such that 
	\begin{equation*}
		\begin{split}
			\frac{1}{\varrho_p(n)}\Big(\sum_{m = 0}^{n-1}\sum_{k = 0}^{2^{m-1}-1}&\frac{\theta_{m,k}^2}{t^m_{2k+1}-t^m_{2k}}+\frac{\theta_{m,k}^2}{t^m_{2k+2}-t^m_{2k+1}}\Big)^{p/2} \le \Big(\sum_{k = 0}^{2^n-1}\sum_{m = 0}^{n-1}\frac{\theta^2_{m,\floor{2^{m-n+1}k} }}{t^{m}_{\floor{2^{m-n}k}+1}-t^{m}_{\floor{2^{m-n}k}}}\Big)^{p/2} \\&\le \varrho_p(n)\Big(\sum_{m = 0}^{n-1}\sum_{k = 0}^{2^{m-1}-1}\frac{\theta_{m,k}^2}{t^m_{2k+1}-t^m_{2k}}+\frac{\theta_{m,k}^2}{t^m_{2k+2}-t^m_{2k+1}}\Big)^{p/2} \quad \text{for } n \in \bN. 
		\end{split}
	\end{equation*}
\end{definition}

To formulate our extension of \Cref{Thm_main}  to unequally spaced partitions, we need an additional condition on our partitions sequence. Let us denote
\begin{equation*}
	\overline{\pi}_n := \sup_{k}|t^n_{k+1} - t^n_k| \quad \text{and} \quad \underline{\pi}_n := \inf_{k}|t^n_{k+1} - t^n_k|.
\end{equation*}
We will say that the sequence $(\bT_n)_{n\in\bN_0}$ is \textit{well-balanced} if there exists a non-decreasing subexponential function $\kappa$ such that 
\begin{equation*}
	\sup_{k \le n} \frac{\overline{\pi}_k}{\underline{\pi}_k} \le \kappa(n),\qquad n\in\bN_0.
\end{equation*}
The  similar, but slightly different concept of a \textit{balanced} partition sequence was introduced earlier by Cont and Das~\cite{ContDas,ContDas21}. 
Furthermore, for $n \in \bN$, we now denote likewise
\begin{equation*}
	s_n := \sqrt{\sum_{m = 0}^{n-1}\sum_{k = 0}^{2^{m}-1}2^{m}\theta^2_{m,k}} \quad \text{and} \quad \wh R_n(x) = 1 - \frac{1}{n}\log_2 s_n.
\end{equation*}
The next theorem shows that the roughness exponent defined over a sequence of well-balanced partitions, its value can be obtained from the sequence $(s_n)_{n \in \bN}$.

\begin{theorem}\label{thm balance}
	Suppose that the generalized Faber--Schauder coefficients of $x$ satisfy the reverse Jensen condition and the partition sequence $(\bT_n)_{n \in \bN_0}$ is well-balanced. Then the function $x$ admits the roughness exponent $R$ if and only if the finite limit $r:= \lim_n \wh R_n(x)$ exists, and in that case we have  $R = r$. 
\end{theorem}

Before proving \Cref{thm balance}, let us show that the main result of dyadic partitions can be directly recovered from \Cref{thm balance}. Suppose $(\bT_n)_{n \in \bN_0}$  is the sequence of dyadic partitions and the function $x \in C[0,1]$ admits the Faber--Schauder coefficients $(\mu_{m,k})$. Then it follows that
\begin{equation*}
	\mu_{m,k} = 2^{m/2+1}\theta_{m,k} \quad \text{and} \quad s_n = \sqrt{\frac{1}{2}\sum_{m = 0}^{n-1}\sum_{k = 0}^{2^m-1}\mu_{m,k}^2},
\end{equation*}
which then reproduces Theorem \ref{Thm_main}. To prepare for the proof of \Cref{thm balance},  we  define for $n\in\bN$, 
\begin{equation}\label{Rademacher fct}
	V_n :=\sum_{k=0}^{2^{n}-1}\bigg(\frac1{t^{n+1}_{2k+1}-t^{n+1}_{2k}}\Ind{[t^{n+1}_{2k},t^{n+1}_{2k+1})} -\frac1{t^{n+1}_{2k+2}-t^{n+1}_{2k+1}}\Ind{[t^{n+1}_{2k+1},t^{n+1}_{2k+2})}\bigg).
\end{equation}
Then $V_n$ is equal to the right-hand derivative of the function 
$\sum_{k=0}^{2^{n}-1}b_{n,k}$.
We consider each $V_n$ as a random variable on the probability space $([0,1],\cF,\bP)$, where $\cF$ is the Borel $\sigma$-field of $[0,1]$ and $\bP$ is the Lebesgue measure on  $\cF$. 
When defining the $\sigma$-field
$$\cF_n:=\sigma\big([t^n_k,t^n_{k+1}):k=0,\dots, 2^n-1\big),
$$
one sees that $V_n$ is $\cF_{n+1}$-measurable. Let us moreover define random variables $\vartheta_n:[0,1]\to\bR$ by
\begin{equation}\label{theta fct}
	\vartheta_n:=\sum_{k=0}^{2^{n}-1}\theta_{n,k}\Ind{[t^n_{k},t^n_{k+1})}.
\end{equation}
Clearly, the random variable $\vartheta_n$ is $\cF_n$--measurable. For each $n \in \bN$, let $x_n$ denote the $n^\text{th}$ truncation of $x$, defined in analogy to  \eqref{truncation eq}. Then
	$$x_{n}(t)=\int_0^t\sum_{m=0}^{n-1}\vartheta_m(s)V_m(s)\,ds.
	$$
	Since $\vartheta_m$ and $V_m$ are constant on intervals of the form $[t^{n}_k,t^{n}_{k+1})$ for $n \ge m$, we have 
	\begin{equation}\label{eq irr sum}
		x(t^n_{k+1})-x(t^n_k)=x_{n}(t^n_{k+1})-x_{n}(t^n_k)=(t^n_{k+1}-t^n_{k})\sum_{m=0}^{n-1}\vartheta_m(t)V_m(t).
	\end{equation}
	Hence, for $p\ge1$, 
	\begin{align*}
		\<x\>_n^{(p)}&=\sum_{k=0}^{2^n-1}\big|x(t^n_{k+1})-x(t^n_k)\big|^p=\sum_{k=0}^{2^n-1}\bigg|(t^n_{k+1}-t^n_{k})\sum_{m=1}^{n}\vartheta_m(t^{n}_k)V_m(t^{n}_k)\bigg|^p\\
		&=
	\sum_{k=0}^{2^n-1}(t^n_{k+1}-t^n_{k})\cdot\bigg|\gamma_{n,p}(t_k)\sum_{m=0}^{n-1}\vartheta_m(t_k)V_m(t_k)\bigg|^p=\bE\bigg[\bigg|\gamma_{n,p}\sum_{m=0}^{n-1}\vartheta_mV_m\bigg|^p\bigg],
	\end{align*}
	where the expectation is taken under the Lebesgue measure $\bP$ and 
	\begin{equation*}
		\gamma_{n,p}(t)  :=\sum_{k=0}^{2^n-1}(t^n_{k+1}-t^n_{k})^{\frac{p-1}{p}}\Ind{[t^n_{k+1},t^n_{k})}(t).
	\end{equation*} 
	 By definition, $\gamma_{n,p}$ is constant on all intervals of the form $[t^{n}_{k},t^{n}_{k+1})$ and hence $\cF_{n}$-measurable. Moreover,  H\"{o}lder's inequality yields that
	\begin{equation}\label{eq irr Holder}
		\bE[\gamma_{n,p}^{-p}]^{-1}\cdot\bE\bigg[\bigg|\sum_{m=0}^{n-1}\vartheta_mV_m\bigg|^{p/2}\bigg]^{2}\le	\bE\bigg[\bigg|\gamma_{n,p}\sum_{m=0}^{n-1}\vartheta_mV_m\bigg|^p\bigg] \le \bE[\gamma_{n,p}^{2p}]^{1/2}\cdot\bE\bigg[\bigg|\sum_{m=0}^{n-1}\vartheta_mV_m\bigg|^{2p}\bigg]^{1/2}.
	\end{equation}
	Now take $M_0:=0$ and 
	$$M_n:=\sum_{m=0}^{n-1}\vartheta_mV_m=\sum_{m=0}^{n-1}\vartheta_m(N_{m+1}-N_m),\qquad n=1,2,\dots.
	$$
	Since for each $m$, the random variable $\vartheta_m$ is  bounded and $\cF_{m}$-measurable, $(M_n)_{n=0,1,\dots}$ is a martingale transform of the martingale $(N_n)_{n=0,1,\dots}$ and hence itself a martingale. 
	We first compute the expectations of the random variable $\gamma_{n,p}$ occurring on both bounds,
	\begin{equation*}
		\begin{split}
			\bE[\gamma_{n,p}^{2p}]^{1/2} &= \Big(\sum_{k = 0}^{2^n-1}(t^n_{k+1}-t^n_{k})^{2(p-1)}\cdot(t^n_{k+1}-t^n_{k})\Big)^{1/2} = \Big(\sum_{k = 0}^{2^n-1}(t^n_{k+1}-t^n_{k})^{2p-1}\Big)^{1/2}.\\ 
			\bE[\gamma_{n,p}^{-p}]^{-1} &= \Big(\sum_{k = 0}^{2^n-1}(t^n_{k+1}-t^n_{k})^{-(p-1)}\cdot(t^n_{k+1}-t^n_{k})\Big)^{-1} = \Big(\sum_{k = 0}^{2^n-1}(t^n_{k+1}-t^n_{k})^{2-p}\Big)^{-1}.
		\end{split}
	\end{equation*}
	Furthermore, the Burkholder inequality yields constants $0 < A_p \le B_p < \infty$, depending only on $p$ but not on $x$, such that
	$$A_p\bE\big[[M]_n^{p}\big]\le \bE\Big[\big|M_n\big|^{2p}\Big]\le B_p\bE\big[[M]_n^{p}\big],
	$$
	where 
	$$[M]_n=\sum_{m=1}^n(M_m-M_{m-1})^2=\sum_{m=0}^{n-1} \vartheta^2_mV^2_m.
	$$
	Note that we can also rephrase the reverse Jensen condition in Definition \ref{def converse Jensen condition} in the same way as in \eqref{Thm_main cond}:
	\begin{equation*}
		\frac{1}{\varrho_p(n)}\bE\big[\sum_{m = 0}^{n-1} V^2_m \vartheta^2_m\big]^{p/2} \le \bE\big[\Big(\sum_{m = 0}^{n-1} V^2_m \vartheta^2_m\Big)^{p/2}\big] \le \varrho_p(n)\bE\big[\sum_{m = 0}^{n-1} V^2_m \vartheta^2_m\big]^{p/2}.
	\end{equation*}
	It then follows from the reverse Jensen condition that
	$$ \frac{1}{\varrho_p(n)}\bE\big[[M]_n\big]^{p} \le \bE\big[[M]_n^{p}\big] \le \varrho_p(n)\bE\big[[M]_n\big]^{p}.
	$$
	Now, we get 
	\begin{equation*}
		\bE\big[[M]_n\big] =	\bE\big[\sum_{m = 0}^{n-1}\vartheta^2_mV^2_m\big] = \sum_{m = 0}^{n-1}\sum_{k = 0}^{2^{m}-1}\frac{\theta_{m,k}^2}{t^{m+1}_{2k+1}-t^{m+1}_{2k}}+\frac{\theta_{m,k}^2}{t^{m+1}_{2k+2}-t^{m+1}_{2k+1}}
	\end{equation*}
	Combining the previous inequalities and the above equality yields
	\begin{equation*}
		\<x\>_n^{(p)} \le C_p\varrho_p(n)\Big(\sum_{m = 0}^{n-1}\sum_{k = 0}^{2^{m}-1}\frac{\theta_{m,k}^2}{t^{m+1}_{2k+1}-t^{m+1}_{2k}}+\frac{\theta_{m,k}^2}{t^{m+1}_{2k+2}-t^{m+1}_{2k+1}}\Big)^{p/2}\cdot\Big(\sum_{k = 0}^{2^n-1}(t^n_{k+1}-t^n_{k})^{2p-1}\Big)^{1/2}.
	\end{equation*}
	By deriving a corresponding lower bound in the same way, we obtain the following proposition. Note that have not yet used the assumption that our partition sequence is well-balanced.

\begin{proposition}\label{thm main 2}
	Let $x \in C[0,1]$ satisfy the reverse Jensen condition. Then for $p \ge 1$, there exist constants $0 < A_p \le B_p < \infty$ and subexponential function $\varrho_p$ depending only on $p$ but not on $x$, such that for all $n \in \bN$, 
	\begin{equation*}
		\begin{split}
			\<x\>_n^{(p)}&\le B_p\varrho_p(n)\Big(\sum_{m = 0}^{n-1}\sum_{k = 0}^{2^{m}-1}\frac{\theta_{m,k}^2}{t^{m+1}_{2k+1}-t^{m+1}_{2k}}+\frac{\theta_{m,k}^2}{t^{m+1}_{2k+2}-t^{m+1}_{2k+1}}\Big)^{p/2}\cdot\Big(\sum_{k = 0}^{2^n-1}(t^n_{k+1}-t^n_{k})^{2p-1}\Big)^{1/2},\\	
			\<x\>_n^{(p)}&\ge \frac{A_p}{\varrho_p(n)}\Big(\sum_{m = 0}^{n-1}\sum_{k = 0}^{2^{m}-1}\frac{\theta_{m,k}^2}{t^{m+1}_{2k+1}-t^{m+1}_{2k}}+\frac{\theta_{m,k}^2}{t^{m+1}_{2k+2}-t^{m+1}_{2k+1}}\Big)^{p/2}\cdot\Big(\sum_{k = 0}^{2^n-1}(t^n_{k+1}-t^n_{k})^{2-p}\Big)^{-1}.
		\end{split}
	\end{equation*}
\end{proposition}

\begin{proof}[Proof of \Cref{thm balance}]
	Since $(\bT_n)_{n \in \bN}$ is a sequence of well-balanced partitions, then $\kappa^{-1}(n)2^{-n} \le t^n_{k+1} - t^n_k \le \kappa(n)2^{-n}$ for $n \in \bN$. Applying this relation to  \Cref{thm main 2} gives
	\begin{equation*}
		\begin{split}
			\<x\>_n^{(p)} &\le B_p\varrho_p(n)\Big(\sum_{m = 0}^{n-1}\sum_{k = 0}^{2^{m}-1}\frac{\theta_{m,k}^2}{t^{m+1}_{2k+1}-t^{m+1}_{2k}}+\frac{\theta_{m,k}^2}{t^{m+1}_{2k+2}-t^{m+1}_{2k+1}}\Big)^{p/2}\cdot\Big(\sum_{k = 0}^{2^n-1}(t^n_{k+1}-t^n_{k})^{2p-1}\Big)^{1/2} \\&\le B_p\varrho_p(n)\Big(\sum_{m = 0}^{n-1}\sum_{k = 0}^{2^{m}-1}2^{m}\kappa(m)\theta_{m,k}^2\Big)^{p/2}\cdot\Big(\sum_{k = 0}^{2^n-1}(2^{-n})^{2p-1}\kappa(n)^{|2p-1|}\Big)^{1/2}\\&\le B_p\varrho_p(n)\kappa(n)^{1+|2p-1|/2}2^{n(1-p)}\Big(\sum_{m = 0}^{n-1}\sum_{k = 0}^{2^{m}-1}2^{m}\theta_{m,k}^2\Big)^{p/2} = B_p\varrho_p(n)\kappa(n)^{1+|2p-1|/2}2^{n(1-p)}s_n^p.
		\end{split}
	\end{equation*}
	Similarly, for the lower bound
	\begin{equation*}
		\<x\>_n^{(p)} \ge \frac{A_p}{\varrho_p(n)\kappa(n)^{p/2+|p-2|}}2^{n(1-p)}s_n^p.
	\end{equation*}
	Taking logarithm on both sides of the above two inequalities and letting $n \ua \infty$ give the result. This completes the proof.
\end{proof}

\noindent{\bf Acknowledgement.} The authors express their gratitude to two anonymous referees for their valuable comments. We thank Zhenyuan Zhang for many enlightening discussions.

	\bibliographystyle{plain}
	\bibliography{CTBook}
	
	\end{document}